\documentclass[10pt,a4]{article}
\usepackage{latexsym}
\usepackage{amsmath}
\usepackage{amssymb}
\usepackage{amsthm}
\usepackage{amscd}
\usepackage{mathrsfs}
\usepackage[all]{xy}

\def\isom
{\simeq\hspace{-1.1em}\raisebox{-1.2ex}{$\rightarrow$}}

\newtheorem{theorem}{Theorem}[section]
\newtheorem{lemma}[theorem]{Lemma}
\newtheorem{corollary}[theorem]{Corollary}
\newtheorem{proposition}[theorem]{Proposition}

\theoremstyle{definition}

\newtheorem{remark}[theorem]{Remark}
\newtheorem{definition}[theorem]{Definition}

\theoremstyle{remark}
\newtheorem{notation}{Notation}

\makeatletter

\@addtoreset{equation}{section}
\makeatother

\def\La{\Lambda}

\newcommand{\cF}{{\cal F}}
\newcommand{\cK}{{\cal K}}
\newcommand{\cH}{{\cal H}}

\renewcommand{\eqref}[1]{(\ref{#1})}

\renewcommand{\bigskip}{\vspace{0.2cm}}
\newcommand{\juq}{{ {j}_! }}

\begin{document}

\title{On localizations of the characteristic classes
of $\ell$-adic sheaves and
  conductor formula in characteristic $p>0$}

\maketitle

\begin{center}
{\sc Takahiro Tsushima}\\
{\it
Graduate School of Mathematical Sciences,
The University of Tokyo, 3-8-1 Komaba, Meguro-ku
Tokyo 153-8914, JAPAN}
\\
E-mail: tsushima@ms.u-tokyo.ac.jp
\end{center}

\begin{abstract}
The Grothendieck-Ogg-Shafarevich formula 
calculates the $\ell$-adic Euler-Poincar\'{e} number
of 
an $\ell$-adic sheaf on a curve
by an invariant produced by the wild ramification of 
the $\ell$-adic sheaf named Swan class.
A.\ Abbes, K.\ Kato and T.\ Saito
generalize this formula to any dimensional scheme
in \cite{KS} and \cite{AS}.
In this paper, assuming the strong resolution of singularities
 we prove a localized
version of a formula proved by A.\ Abbes and T.\ Saito
in \cite{AS} using the characteristic class of an
$\ell$-adic sheaf.
As an application, we prove a conductor formula in equal characteristic.
\end{abstract}

\section{Introduction}
\noindent
The Grothendieck-Ogg-Shafarevich formula 
calculates the $\ell$-adic Euler-Poincar\'{e} number
of 
an $\ell$-adic sheaf on a curve
by an invariant produced by the wild ramification of 
the $\ell$-adic sheaf named Swan class.
Generalizations of this formula to
the surface case are done by Deligne, Kato and Laumon. 
Recently this formula is generalized
to any dimensional scheme by 
A.~Abbes, K.~Kato and T.~Saito in \cite{AS} and \cite{KS}.
To generalize this formula to any dimensional
case, K.~Kato and T.~Saito defines the Swan class
of an $\ell$-adic sheaf on any dimensional scheme
using alteration and logarithmic blow-up.
A.~Abbes and T.~Saito rediscovered the characteristic
class of an $\ell$-adic sheaf using the Verdier pairing (SGA5) and studied 
its properties in \cite{AS}. 
The characteristic
class of an $\ell$-adic sheaf
on a scheme 
is 
a cohomological element in the top \'{e}tale cohomology group
which goes to
the $\ell$-adic Euler-Poincar\'{e} number under the trace map
in the case where the scheme is proper.
They calculate the characteristic
class by the Swan class defined by K.~Kato and T.~Saito.
This formula is a refinement of the results
proved by Kato and Saito in \cite{KS}.
We call this refinement the Abbes-Saito formula.
\\In this paper, 
we prove a localized version of the Abbes-Saito formula,
a localized version of the Lefschetz-Verdier trace formula and a refinement of the Kato-Saito conductor formula
in equal characteristic.

A localized version of the Abbes-Saito formula is an equality (Theorem \ref{43}) of the localized characteristic
class of a smooth $\ell$-adic sheaf and the Swan class in  an \'{e}tale
cohomology group with support.
 We call this equality the localized Abbes-Saito
formula. 
To show the localized Abbes-Saito formula, we generalize the
localized characteristic class of a smooth $\ell$-adic sheaf also
defined in \cite{AS} to the localized characteristic class 
of an $\ell$-adic sheaf
with a
cohomological correspondence and prove its compatibility with pull-back. 
Assuming
the strong resolution of singularities,
as a direct consequence of this compatibility,
we prove the localized Abbes-Saito formula in Theorem 
\ref{43}.

We prove the compatibility of the
localized chracteristic class with proper push-forward
in Proposition \ref{prop:2}.
This is a localized version of the Lefschetz-Verdier trace formula,
which we call the localized Lefschetz-Verdier trace formula.

As an application of the localized Abbes-Saito formula and the localized
Lefschetz-Verdier trace formula, we prove a conductor formula
 in equal characteristic in Corollary \ref{COroo}. The conductor
formula calculates the Swan conductor of 
an $\ell$-adic representation which appears when we consider
a fibration on a curve
by the Swan class of an $\ell$-adic sheaf defined by Kato and Saito
in \cite{KS}.
We call this conductor formula the Kato-Saito conductor formula
in characteristic $p>0$. To prove
this formula is the main purpose to consider localizations.
To refine the Kato-Saito conductor formula,
we define a localization as a cohomology class with
support on the wild locus in subsection 3.3 which we call
the logarithmic localized characteristic class
and prove its compatibility with proper
push-forward in Theorem \ref{tth}.
In \cite{T}, we prove a refinement of the formula
for a smooth sheaf of rank 1
proved by Abbes-Saito in \cite{AS} using an idea of T.\ Saito in \cite{S}.
As an application of Theorem \ref{tth} and the result in \cite{T},
we prove a refinement of the Kato-Saito conductor formula
for a smooth sheaf of rank 1
in Corollary \ref{rank1}.

It is a great pleasure for the author to thank Prof.\ T.~Saito for suggesting the problem and the idea of 
the proof of the localized Abbes-Saito formula.
Prof.\ T.\ Saito suggested to the author that the localized
Abbes-Saito formula and
the localized Lefschetz-Verdier trace formula imply the Kato-Saito
conductor formula in equal characteristic after the author proved a special case of the compatibility
of the localized characteristic class with proper push-forward
and the localized Abbes-Saito formula. 
The author would like to thank Prof.\ T.\ Saito for many suggestions, encouragements
 and for pointing out an error of a proof of
Lemma \ref{7} in an early version of this paper and giving him a useful
suggestion to improve this lemma. 
The author would like to express his sincere gratitude 
to Prof.\ Ahmed Abbes for stimulating discussions in the Tokyo University. 
The author would like to acknowledge
the hospitability of Prof.\ Deninger
in Muenster University where a part of this work was achieved. 
This research is supported by JSPS-Fellowships for Young Scientists.
\begin{notation}
In this paper,
$k$ denotes a field.
Schemes over $k$ are assumed to be separated and of finite type.
For a divisor with simple normal crossings
of a smooth scheme over $k$,
we assume that the irreducible components and their intersections are also
smooth over $k.$
The letter $l$ denotes a prime number invertible in $k$
and $\La$
denotes a finite commutative $\mathbb{Z}_l$-algebra.
For a scheme $X$ over $k$,
$D_{\rm ctf}(X)$ denotes the derived category of complexes of $\La$-modules
of finite tor-dimension on the \'{e}tale site of $X$
with constructible cohomology.
Let $\cK_X$ denote $Rf^!\La$
where $f:X \longrightarrow {\rm Spec} k$ is the structure
map and let ${\bf D}$ denote the functor $R\mathcal{H}om(\ ,\cK_X)$.
For objects $\cF$ and $\mathcal{G}$ of $D_{\rm ctf}(X)$
and $D_{\rm ctf}(Y)$ on schemes $X$ and $Y$
over $k$, $\cF \boxtimes \mathcal{G}$ denotes ${\rm pr}_1^\ast \cF \otimes {\rm pr}_2^\ast \mathcal{G}$ on $X \times Y.$
When  we say a scheme $X$ is of dimension $d$, we understand
that every irreducible component of $X$ is of dimension $d.$
\end{notation}\section{Review of the characteristic class etc.}
\subsection{Cohomological correspondence and the evaluation map}
We recall the definition and some properties of a cohomological correspondence needed in this paper from \cite[subsection 1.2]{AS}.
\begin{definition}\cite[Definition 1.2.1]{AS}
Let $X$ and $Y$ be schemes over $k$ and $\mathcal{F}$ and $\mathcal{G}$
 be objects of $D_{\rm ctf}(X)$ and of $D_{\rm ctf}(Y)$ respectively. We call a correspondence between $X$ and $Y$ a scheme $C$ over $k$ and morphisms $c_1:C \longrightarrow X$ and
$c_2:C \longrightarrow Y$ over $k.$ We put $c=(c_1,c_2):C \longrightarrow X \times Y$ the corresponding morphism.
We call a morphism $u:c_2^\ast \mathcal{G} \longrightarrow Rc_1^!\mathcal{F}$ a
cohomological correspondence from $\mathcal{G}$
to $\mathcal{F}$ on $C$. \end{definition}
We identify the cohomological
correspondence with a section of
$$H_{C}^0(X \times Y,R\mathcal{H}om({\rm pr}_2^\ast \mathcal{G},R{\rm pr}_1^! \mathcal{F}))
\simeq H^0(C,{\cH}om(c_2^\ast \mathcal{G},Rc_1^! \cF)).$$
A typical example of a cohomological correspondence is given as follows.
Assume $X$ and $Y$ are smooth of dimension $d$ over $k$
and $c=(c_1,c_2):C \longrightarrow X \times Y$
is a closed immersion.
Let $\mathcal{F}$ and $\mathcal{G}$ be sheaves of free $\La$-modules on
$X$ and $Y$ respectively and assume that $\mathcal{G}$ is smooth.
Then, the canonical map
$c^\ast \mathcal{H}om({\rm pr}_2^\ast \mathcal{G},{\rm pr}_1^\ast \mathcal{F})
\longrightarrow
\mathcal{H}om(c_2^\ast \mathcal{G},c_1^\ast \mathcal{F})$
is an isomorphism and we identify
${\rm Hom}(c_2^\ast \mathcal{G},c_1^\ast \mathcal{F})=
\Gamma(C,c^\ast \mathcal{H}om({\rm pr}_2^\ast \mathcal{G},{\rm pr}_1^\ast \mathcal{F}))$.
Since ${\rm pr}_1:X \times Y \longrightarrow X$
is smooth, we have a canonical isomorphism
${\rm pr}_1^\ast \mathcal{F}(d)[2d]=R{\rm pr}_1^! \mathcal{F}$
and we identify $R\mathcal{H}om({\rm pr}_2^\ast \mathcal{G},R{\rm pr}_1^! \mathcal{F})
=
\mathcal{H}om({\rm pr}_2^\ast \mathcal{G},{\rm pr}_1^\ast \mathcal{F})(d)[2d].$
Then the cycle class map $CH_{d}(C) \longrightarrow H_{C}^{2d}(X \times Y,\La(d))$ induces a pairing
$$CH_{d}(C) \otimes {\mathrm{Hom}}(c_2^\ast \mathcal{G},c_1^\ast \mathcal{F}) \longrightarrow
H_{C}^{2d}(X \times Y,\La(d)) \otimes
\Gamma(C,c^\ast \mathcal{H}om({\rm pr}_2^\ast \mathcal{G},{\rm pr}_1^\ast \mathcal{F}))$$
$$\longrightarrow H_{C}^0(X \times Y,R\mathcal{H}om({\rm pr}_2^\ast \mathcal{G},R{\rm pr}_1^! \mathcal{F})).\!\!\!\!\!\!\!\!\!\!\!\!\!\!\!\!\!\!\!\!\!\!\!\!\!\!\!\!\!\!\!\!\!\!\!\!\!\!\!\!\!\!\!\!$$
In other words, the pair $(\Gamma,\gamma)$ of a cycle class $\Gamma \in
CH_{d}(C)$
and a homomorphism
$\gamma:c_2^\ast \mathcal{G} \longrightarrow c_1^\ast \mathcal{F}$
defines a cohomological correspondence $u(\Gamma,\gamma).$
\\We recall the definition of the push-forward of a cohomological correspondence. We consider the commutative diagram
\begin{equation}\label{2-1}
\xymatrix{
X \ar[d]_{f} & C \ar[l]_{c_1} \ar[r]^{c_2}\ar[d]_{h} & Y \ar[d]^{g} \\
X' & C' \ar[l]_{c'_1}\ar[r]^{c'_2} & Y' \\
}
\end{equation}
of schemes over $k$. A canonical isomorphism
$$R(f \times g)_\ast R\mathcal{H}om({\rm pr}_2^\ast \mathcal{G},R{\rm pr}_1^! \mathcal{F}) \longrightarrow R\mathcal{H}om({\rm pr}_2^\ast Rg_!\mathcal{G},R{\rm pr}_1^! Rf_\ast \mathcal{F})$$
is defined in \cite[(3.3.1)]{Gr}, using the isomorphism
$$ R\mathcal{H}om({\rm pr}_2^\ast \mathcal{G},R{\rm pr}_1^! \mathcal{F}) \longrightarrow \mathcal{F} \boxtimes^{L} {\bf D}\mathcal{G}$$ defined in \cite[(3.1.1)]{Gr}.
In the diagram \eqref{2-1}, we assume that the vertical arrows are proper. The above diagram defines a commutative diagram
\[\xymatrix{
C \ar[r]^{\!\!\!\!\!\!\!\!\!c}\ar[d]^{h} & X \times Y \ar[d]^{f \times g} \\
C' \ar[r]^{\!\!\!\!\!\!\!\!\!c'} & X' \times Y'.\\
}
\]
Let $u:c_2^\ast \mathcal{G} \longrightarrow Rc_1^!\mathcal {F}$
 be a cohomological correspondence. We identify $u$ with a map $u:\La \longrightarrow Rc^!R\mathcal{H}om({\rm pr}_2^\ast \mathcal{G},R{\rm pr}_1^!\mathcal{F})$.
Then, it induces a map $\La \longrightarrow Rh_\ast \La \longrightarrow Rh_\ast Rc^!R\mathcal{H}om({\rm pr}_2^\ast \mathcal{G},R{\rm pr}_1^!\mathcal{F})$
where the first map
$\La \longrightarrow Rh_\ast \La$ is the adjunction of the identity.
By the assumption that $f,g$ and $h$ are proper, the base change map defines a map of functors $Rh_\ast Rc^!=Rh_!Rc^! \longrightarrow Rc'^!R(f \times g)_\ast$.
By composing them with the isomorphism
$$R(f \times g)_\ast R\mathcal{H}om({\rm pr}_2^\ast \mathcal{G},R{\rm pr}_1^! \mathcal{F}) \longrightarrow R\mathcal{H}om({\rm pr}_2^\ast Rg_!\mathcal{G},R{\rm pr}_1^! Rf_\ast \mathcal{F}),$$ we obtain a map
$$\La \longrightarrow Rh_\ast Rc^!R\mathcal{H}om({\rm pr}_2^\ast \mathcal{G},R{\rm pr}_1^!\mathcal{F})$$ 
$$\longrightarrow Rc'^!R(f \times g)_\ast R\mathcal{H}om({\rm pr}_2^\ast \mathcal{G},R{\rm pr}_1^!\mathcal{F}) \longrightarrow Rc'^!R\mathcal{H}om({\rm pr}_2^\ast Rg_!\mathcal{G},R{\rm pr}_1^! Rf_\ast \mathcal{F}).$$
We define the push-forward $h_\ast u:c_2^\ast Rg_! \mathcal{G}=
{c'_2}^\ast Rg_\ast {\mathcal{G}} \longrightarrow {Rc'_1}^! Rf_\ast {\mathcal{F}}$ of $u$ to be the corresponding cohomological correspondence. 
The push-forward $h_\ast u$ is equal to the composition of the maps
$${c'_2}^\ast Rg_\ast \mathcal{G} \longrightarrow Rh_\ast c_2^\ast \mathcal{G} \longrightarrow Rh_\ast Rc_1^! \mathcal{F} \longrightarrow {Rc'_1}^! Rf_\ast \mathcal{F}$$
where the first and the third maps are the base change maps and the second map is the push-forward $h_\ast u$.

We consider the commutative diagram
\begin{equation}\label{2}
\xymatrix{
U \ar[d]_{j_U} & C \ar[r]^{c_2} \ar[d]^{j_C} \ar[l]_{c_1} & V \ar[d]^{j_V} \\
X & \bar{C} \ar[r]^{\bar{c}_2} \ar[l]_{\bar{c}_1} & Y\\
}
\end{equation}
of schemes over $k$ where the vertical arrows are open immersions. Let $\mathcal{F}$ and $\mathcal{G}$ be objects of $D_{\rm ctf}(X)$ and $D_{\rm ctf}(Y)$
respectively and $\bar{u}:\bar{c}_2^\ast \mathcal{G} \longrightarrow R\bar{c}_1^! \mathcal{F}$ be a cohomological correspondence on $\bar{C}$.
 Let $\cF_U=j_{U}^\ast \cF$ and $\mathcal{G}_V=j_V^\ast \mathcal{G}$ be the restrictions. We identify $j_C ^\ast R\bar{c}_1^!\cF=j_C ^! R\bar{c}_1^!\cF=Rc_1^!\cF_U$ by the composite isomorphism. Then, the restriction $j_C^\ast \bar{u}$ on $C$ defines a cohomological correspondence $u:c_2^\ast \mathcal{G}_V=j_C^\ast \bar{c}_2^\ast \mathcal{G} \longrightarrow j_C^\ast R\bar{c}_1^!\mathcal{F}=Rc_1^!\cF_U$.

We recall the zero-extension of a cohomological correspondence playing an important role when we define a refined (localized) characteristic class.

\begin{lemma}{\rm \cite[lemma 1.2.2]{AS}}
Let the notation be as above and let $j:U \times V \longrightarrow
 X \times Y $ be the product $j_{U} \times j_{V}$.
We put
$\mathcal{H}=R\mathcal{H}om({\rm pr}_2 ^\ast \mathcal{G}_V,R{\rm pr}_1^!\mathcal{F}_U )$
on
$U \times V$ and
$\bar{\cH}=R\mathcal{H}om({\rm pr}_2 ^\ast \mathcal{G},R{\rm pr}_1^!\mathcal{F} )$
on
$X \times Y.$
We identify a cohomological correspondence
$\bar{u}:\bar{c}_2^\ast \mathcal{G} \longrightarrow R\bar{c}_1^!\mathcal{F}$
with a section $\bar{u}$ of
${R\bar{c}}^! \bar{\mathcal{H}}$
and the associated map
$\bar{u}:\bar{c}_! \La \longrightarrow \bar{\mathcal{H}}.$
We also identify the restriction
$u=j_C ^\ast \bar{u}:c_2 ^\ast \mathcal{G}_V \longrightarrow Rc_1^! \mathcal{F}_U$
with a section $u$ of $Rc^! \mathcal{H}$
and the associated map
$u:c_! \La \longrightarrow \mathcal{H}.$
Then, we have the following.
\\1.
The section $u$ of $Rc^! \mathcal{H}$ is the image of the restriction of
$\bar{u}$ by the composition isomorphism
$j_C^\ast {R\bar{c}}^! \bar{\mathcal{H}} \longrightarrow Rc^!j^! \bar{\mathcal{H}}=Rc^!j^ \ast \bar{\mathcal{H}}=Rc^!\mathcal{H}.$
\\2.
The square
\[\xymatrix{
j^\ast \bar{c}_!\La \ar[r]^{j^\ast \bar{u}} &
j^\ast \bar{\mathcal{H}} \ar[d]\\
c_! \La \ar[u]\ar[r]^{u} & \mathcal{H}  \\
}
\]
is commutative.

\end{lemma}
\begin{lemma}{\rm \cite[lemma 1.2.3]{AS}}
Assume that the right square in the diagram (\ref{2}) is cartesian. Let $\cF$ and
 $\mathcal{G}$ be objects of $D_{\rm ctf}(U)$ and of $D_{\rm ctf}(V)$ respectively
and $u:c_2^\ast \mathcal{G} \longrightarrow Rc_1^!\cF$ be a cohomological correspondence on $C$. Then, there exists a unique cohomological correspondence $\bar{u}:c_2^\ast {j_V}_!\mathcal{G} \longrightarrow R\bar{c}_1^!{j_U}_!\mathcal{F}$ on $\bar{C}$ such that $j_C^\ast \bar{u}=u.$
\end{lemma}
\begin{corollary}{\rm \cite[Corollary 1.2.4]{AS}}\label{zero}
1. Assume that the map
$c_2:C \longrightarrow V$
is proper and
$C$
is dense in $\bar{C}.$
Then the right square in the diagram (\ref{2})
is cartesian.
\\2.
Assume that the right square in the diagrm (\ref{2}) is cartesian.
Let
$\mathcal{F}$
and
$\mathcal{G}$
be objects of $D_{\rm ctf}(U)$ and of $D_{\rm ctf}(V)$ respectively and
$u:c_2 ^\ast \mathcal{G} \longrightarrow Rc_1^! \mathcal{F}$
be a cohomological correspondence on
$C.$
Then,
$\bar{u}={j_C}_! u:
\bar{c}_2^\ast {j_V}_! \mathcal{G} \longrightarrow
R\bar{c}_1^! {j_U}_! \mathcal{F}$
is the unique cohomological correspondence on $\bar{C}$
such that $j_C^\ast \bar{u}=u.$
\end{corollary}
We call
${j_C}_!u:\bar{c}_2^\ast {j_V}_! \mathcal{G} \longrightarrow
R\bar{c}_1^! {j_U}_! \mathcal{F}$
the zero-extension of $u$.

We define the pull-back of a cohomological correspondence. Let $f:X' \longrightarrow X$ and $g:Y' \longrightarrow Y$ be morphisms of smooth schemes over $k$.
We assume ${\rm dim}\ X$=${\rm dim}\ X'$ and ${\rm dim}\ Y$=${\rm dim}\ Y'.$ Then the canonical map
$f^\ast \longrightarrow Rf^!$ and the isomorphism
$$Rc^!R\mathcal{H}om({\rm pr}_2^\ast \mathcal{G},R{\rm pr}_1^! \mathcal{F})
\longrightarrow
R\mathcal{H}om(c_2^\ast \mathcal{G},Rc_1^! \mathcal{F})$$ induce a map
{\small
$$(f \times g)^\ast R\mathcal{H}om({\rm pr}_2^\ast \mathcal{G},R{\rm pr}_1^! \mathcal{F})
\longrightarrow
R(f \times g)^!R\mathcal{H}om({\rm pr}_2^\ast \mathcal{G},R{\rm pr}_1^! \mathcal{F}) \longrightarrow R\mathcal{H}om({\rm pr'}_2^\ast g^\ast \mathcal{G},R{\rm pr'}_1^! Rf^!\mathcal{F}).$$}
With the isomorphism
$$ R\mathcal{H}om({\rm pr}_2^\ast \mathcal{G},R{\rm pr}_1^! \mathcal{F})
\longrightarrow
\mathcal{F} \boxtimes^{L} {\bf D}\mathcal{G},$$
the above composition 
is identified with the composition
$$(f \times g)^\ast(\mathcal{F} \boxtimes^{L} {\bf D}\mathcal{G})
\longrightarrow
f^\ast \cF \boxtimes^{L} g^\ast {\bf D}\mathcal{G} \longrightarrow
Rf^!\cF \boxtimes^{L} Rg^!{\bf D}\mathcal{G} \longrightarrow
Rf^!\cF \boxtimes^{L} {\bf D}g^\ast \mathcal{G}.$$
Let $c=(c_1,c_2):C \longrightarrow X \times Y$ be a correspondence and
$u:c_2^\ast \mathcal{G} \longrightarrow Rc_1^! \cF$ be a cohomological correspondence on $C$. We identify $u$ with a map $u:c_! \La \longrightarrow R\mathcal{H}om({\rm pr}_2^\ast \mathcal{G},R{\rm pr}_1^!\cF)$ as above. We define a correspondence $c'=(c'_1,c'_2):C' \longrightarrow X' \times Y'$ by the cartesian diagram
\[\xymatrix{
C' \ar[r]^{\!\!\!\!\!\!\!\!c'}\ar[d]^{h} & X' \times Y'\ar[d]^{f \times g} \\
C \ar[r]^{\!\!\!\!\!\!\!\!c} & X \times Y. \\
}
\]
By the proper base change theorem, the base change map
$(f \times g)^\ast c_! \La \longrightarrow {c'}_! \La$ is an isomorphism.
Hence the map
$u:c_!\La \longrightarrow R\mathcal{H}om({\rm pr}_2^\ast\mathcal{G}, R{\rm pr}_1^!\cF)$ induces a map
$$c'_!\La \simeq (f \times g)^\ast c_!\La \longrightarrow
(f \times g)^\ast R\mathcal{H}om({\rm pr}_2^\ast\mathcal{G},R{\rm pr}_1^! \mathcal{F}) \longrightarrow
R\mathcal{H}om({{\rm pr'}_2}^\ast g^\ast \mathcal{G},R{{\rm pr'}_1}^! Rf^!\mathcal{F}).$$
The composition defines a cohomological correspondence
$(f \times g)^\ast u:{c'_2}^\ast g^\ast \mathcal{G} \longrightarrow R{c'_1}^!Rf^!\cF=R{c'_1}^!f^\ast \cF.$
We call $(f \times g)^\ast u$ the pull-back of $u$ by $f \times g.$

We recall an evaluation map from \cite[subsection 2.1]{AS}.
Let $X$ be a scheme over $k$ and $\delta:X=\Delta_{X} \longrightarrow
X \times X$ be the diagonal map.
Let $\mathcal{F}$ be an object of $D_{\rm ctf}(X)$
and let $1=u(X,1)$ the cohomological correspondence defined by the identity of $\cF$ on the diagonal $X$. An isomorphism
$$\mathcal{H}=R\mathcal{H}om({\rm pr}_2^\ast \mathcal{F},R{\rm pr}_1^! \mathcal{F})
\longrightarrow \mathcal{F} \boxtimes^{L} {\bf D}\mathcal{F}$$
induces an isomorphism
$$\delta^\ast \mathcal{H} \longrightarrow \mathcal{F} \otimes^{L} {\bf D}\mathcal{F}.$$
Thus the evaluation map
$$\mathcal{F} \otimes^{L} {\bf D}\mathcal{F}
\longrightarrow
\mathcal{K}_X$$
induces a map
\begin{equation}\label{delta1}e:\delta^\ast \mathcal{H} \longrightarrow
\mathcal{K}_X.\end{equation}
We call this map the evaluation map of $\mathcal{F}$.
We define another evaluation map.
Let $X$ be a scheme over $k$ and $j:U \longrightarrow X$ be an open immersion over $k.$
Let
$\delta_X:X \longrightarrow X \times X$ and
$\delta_U:U \longrightarrow U \times U$ denote the diagonal maps.
Let $\mathcal{F}$ be an object of $D_{\rm ctf}(U)$.
We put $\mathcal{H}=R\mathcal{H}om({\rm pr}_2^\ast \mathcal{F},R{\rm pr}_1^! \mathcal{F})$
on $U \times U$ and
$\bar{\mathcal{H}}=R\mathcal{H}om({\rm pr}_2^ \ast {j}_! \mathcal{F},R{\rm pr}_1^! {j}_! \mathcal{F})$
on
$X \times X$ respectively.
Since we have
$\bar{\mathcal{H}}=({j} \times 1)_!R{(1 \times {j})}_\ast \mathcal{H}$
by the Kunneth formula,
we obtain a canonical isomorphism
${j}_!\delta_{U}^\ast \mathcal{H} \simeq
\delta^\ast \bar{\mathcal{H}}.$
Thus the evaluation map of $\mathcal{F}$ on $U$ defined as above
$$e:\delta_{U}^\ast \mathcal{H} \longrightarrow \mathcal{K}_U$$
induces an evaluation map
\begin{equation}\label{delta2}
{j}_!e:\delta^\ast \bar{\mathcal{H}}
\longrightarrow
{j}_! \mathcal{K}_U.\end{equation}

\subsection{Characteristic class of a $\La$-sheaf
with a cohomological correspondence}
We briefly recall the definition of the (refined) characteristic class of
a $\La$-sheaf.(c.f.\ \cite[Definition 2.1.8]{AS}.)
Let $X$ be a scheme over $k$,
$U$ an open subscheme, $j:U \longrightarrow X$ the open immersion
and
$\delta:X=\Delta_{X} \longrightarrow
X \times X$ the diagonal map.
Let $C$ be a closed subscheme of $U \times U$,
$\bar{C}$ the closure of $C$ in $X \times X,$ and
 $c:C \longrightarrow U \times U$ and
$\bar{c}:\bar{C} \longrightarrow X \times X$
the closed immersions. Let $j_C:C \longrightarrow \bar{C}$
denote the open immersion.
We assume that $C=(X \times U) \cap \bar{C}.$
Let $\mathcal{F}$ be an object of $D_{\rm ctf}(U).$
We put $\bar{\mathcal{H}}:=R\mathcal{H}om({\rm pr}_2^ \ast {j}_! \cF,R{\rm pr}_1^! {j}_! \cF)$ on $X \times X.$
Let $u$ be a cohomological correspondence of $\mathcal{F}$
on $C.$ We have the zero-extension ${j_C}_!u$
of $u$ by Corollary \ref{zero}. We identify
the cohomological correspondence ${j_C}_!u$ with a section
$${j_C}_!u \in H_{\bar{C}}^0(X \times X, \bar{\mathcal{H}}).$$
The pull-back by $\delta$ and the evaluation map (\ref{delta2})
induce
$${j_C}_!u \in H_{\bar{C}}^0(X \times X, \bar{\mathcal{H}})
\longrightarrow
H_{\bar{C} \cap X}^0(X,\delta ^\ast \bar{\mathcal{H}})
\longrightarrow H_{\bar{C} \cap X}^0(X,{j}_!\cK_U).$$
The image of ${j_C}_!u$ under the composite defines a cohomology class in $H_{\bar{C} \cap X}^0(X,j_!\cK_U).$
We denote it by $C_!({j}_!\cF,\bar{C},{j_C}_!u)$
and call it {\it the refined characteristic class of ${j}_!\cF$
with a cohomological correspondence ${j_C}_!u$ on $\bar{C}.$}
We define {\it the characteristic class
$C({j}_!\cF,\bar{C},{j_C}_!u) \in H_{\bar{C} \cap X}^0(X,\cK_X)$}
to be the image of $C_!({j}_!\cF,\bar{C},{j_C}_!u)$ under
 the canonical map
$H_{\bar{C} \cap X}^0(X,j_!\cK_U) \longrightarrow H_{\bar{C} \cap X}^0(X,\cK_X).$

If $\bar{C}$
is the diagonal $\delta(X) \subset X \times X$ and $u:\cF \longrightarrow \cF$
is an endomorphism, we drop $\bar{C}$ from the notation and simply write
$C_!({j}_!\cF,j_!u) \in H^0(X,{j}_!\cK_U)$ for the refined characteristic class
and $C({j}_!\cF,j_!u) \in H^0(X,\cK_X)$ for the characteristic class respectively.
Further, if $u$ is the identity, we simply write
$C_!({j}_!\cF)$ (resp.\ $C(j_!\cF)$) and call it the refined characteristic class of ${j}_!\cF.$ 
(resp.\ the characteristic class of ${j}_!\cF.$)

\section{Refined localized characteristic class}

\subsection{Refined localized characteristic class}
We will define a localized version of the (refined) characteristic class of a $\La$-sheaf
with a cohomological correspondence. 
Let $X$ be a scheme over a field $k$ and
$U \subseteq X$  an open dense subscheme smooth of dimension $d$ over $k$, $S:=X \backslash U$ the complement,
$j : U \longrightarrow X$ the open immersion, and $\delta_U:U \longrightarrow U \times U$
 and $\delta_X:X \longrightarrow X \times X$ the diagonal closed immersions.
Let $C$ be a closed
subscheme of $U \times U$,
 $\bar{C}$ the closure of $C$ in $X \times X$,
 $\bar{c}:\bar{C} \longrightarrow X \times X$
 the closed immersion, and
 ${g_{\bar{C}}}:X \times X \backslash \bar{C} \longrightarrow X \times X$
 and $j_{C}:C \longrightarrow \bar {C}$ the open immersions.
We assume $C=\bar{C} \cap (X \times U)$.( c.f.\ Lemma 2.3.)
 
Let $\cF$ be a smooth $\La$-sheaf on $U.$ We put 
$\bar{\cH}:=R \mathcal{H}om({\rm pr}_2^\ast{j}_!\cF,R{\rm pr}_1^!{j}_!\cF)$ on $X \times X.$
The canonical map $\La
\longrightarrow R{g_{\bar{C}}}_\ast \La$
induces a map
$$H_{\bar{C}}^0(X \times X, \bar{\mathcal{H}}) \longrightarrow
H_{\bar{C}}^0(X \times X,\bar{\mathcal{H}} \otimes R{g_{\bar{C}}}_\ast \La).$$
\begin{lemma}\label{lll1}
The canonical map
$H_{\bar{C} \backslash C}^0(X \times X,\bar{\mathcal{H}} \otimes
R{g_{\bar{C}}}_\ast \La) \longrightarrow H_{\bar{C}}^0(X \times
X,\bar{\mathcal{H}} \otimes R{g_{\bar{C}}}_\ast \La)$ is an
isomorphism.
\end{lemma}
\begin{proof}We put $\cH_U:=R{\cH}om({\rm pr}_2^\ast \cF,R{\rm pr}_1^!\cF)$ on $U \times U.$
By the localization sequence, it is sufficient to prove that $H^i_{C}(U \times U,\cH_U
\otimes {R{g_C}}_\ast \La)=0$ for all $i$
where $g_C:U \times U \backslash C \longrightarrow U \times U$
denotes
 the open immersion. Since $\cF$ is a smooth sheaf on $U$, the canonical
map $\cH_U \otimes {Rg_C}_\ast \La \longrightarrow {Rg_C}_\ast
g_{C}^\ast \cH_U$ is an isomorphism by the projection formula. Therefore
 we obtain isomorphisms $H_C^i(U \times
U,\cH_U \otimes {Rg_C}_\ast \La) \simeq H^i(C,Rc^!{Rg_C}_\ast
g_{C}^\ast \cH_U) \simeq 0$ 
where $c:C \longrightarrow U \times U$ is the closed immersion. 
Hence the assertion follows.
\end{proof}The pull-back by $\delta_X$ and
the evaluation map (\ref{delta2})
induce a map
$$e \cdot \delta_X ^\ast:H_{\bar{C} \backslash C}^0(X \times X,\bar{\mathcal{H}} \otimes R{g_{\bar{C}}}_\ast \La) \longrightarrow
H_{\bar{C} \cap S}^0(X, \juq \cK_U \otimes \delta_X^\ast R{g_{\bar{C}}}_\ast \La).
$$
We have obtained the maps
\begin{align}\label{definition loc}
\xymatrix{
H_{\bar{C}}^0(X \times X,\bar{\mathcal{H}} \otimes R{g_{\bar{C}}}_\ast \La)
&
H_{\bar{C} \backslash C}^0(X \times X,\bar{\mathcal{H}} \otimes R{g_{\bar{C}}}_\ast \La) 
\ar[l]^{\simeq}_{\!\!\!{\small {\rm Lemma} \ref{lll1}}}\ar[d]^{e \cdot \delta_X ^\ast} \\
{j_C}_!u \in H_{\bar{C}}^0(X \times X, \bar{\mathcal{H}}) \ar[u]^{\rm can.} & H_{\bar{C} \cap S}^0(X, \juq \cK_U \otimes \delta_X^\ast R{g_{\bar{C}}}_\ast \La).\\
}
\end{align}
We write
 \begin{equation}\label{loc map}
{\rm loc}_{X,C,\cF}:H_{\bar{C}}^0(X \times X, \bar{\mathcal{H}}) \longrightarrow H_{\bar{C} \backslash C}^0(X \times X,\bar{\mathcal{H}} \otimes R{g_{\bar{C}}}_\ast \La)
\end{equation}
for the composite
$H_{\bar{C}}^0(X \times X, \bar{\mathcal{H}}) \longrightarrow 
H_{\bar{C}}^0(X \times X,\bar{\mathcal{H}} \otimes R{g_{\bar{C}}}_\ast \La)
\simeq 
H_{\bar{C} \backslash C}^0(X \times X,\bar{\mathcal{H}} \otimes R{g_{\bar{C}}}_\ast \La).$
\begin{definition}\label{refined}
Let $u$ be a cohomological correspondence of $\mathcal{F}$ on $C$ and
 ${j_C}_! u$ the zero-extension of the cohomological correspondence $u$ recalled in subsection 2.1. 
The image of the element 
${\rm loc}_{X,C,\cF}({j_C}_!u) \in H_{\bar{C} \backslash C}^0(X \times X, \bar{\mathcal{H}} \otimes R{g_{\bar{C}}}_\ast \La)$ by the map 
$e \cdot \delta_X ^\ast$ defines 
a cohomology class in $H_{\bar{C} \cap S}^0(X, \juq \cK_U \otimes \delta_X^\ast R{g_{\bar{C}}}_\ast \La)$ and denotes
$$ C_{S,!}^0({j}_! \cF,\bar{C},{j_C}_!u) \in
 H_{\bar{C} \cap S}^0(X, \juq \cK_U \otimes \delta_X^\ast R{g_{\bar{C}}}_\ast \La).$$ We call this element
{\it the refined localized characteristic class
of ${j}_!\cF$ with a cohomological correspondence
${j_C}_!u$ on $\bar{C}$.}
If $C=\delta_U(U)$ is the diagonal and $u:\cF \longrightarrow \cF$
is an endomorphism, we drop $C$ from the notation and we write
$C_{S,!}^0({j}_! \cF,j_!u) \in H_{S}^0(X, \juq \cK_U \otimes \delta_X^\ast R{g_{X}}_\ast \La).$
Further if $u$ is the identity, we simply write $C_{S,!}^0({j}_!\cF).$
\end{definition}
We assume that $C$ is smooth purely of dimension $d$ over $k$ in the following. We have a distinguished triangle
$${c}_\ast R{c}^!\La \longrightarrow
\La \longrightarrow
R{g_C}_\ast \La \longrightarrow.$$
Since $C$ is smooth over $k$,
the cycle class $[C]$ defines an isomorphism
$\La(-d)[-2d] \isom Rc^! \La$
by the purity theorem.
Therefore we acquire a distinguished triangle
$c_\ast \La(-d)[-2d] \longrightarrow
\La \longrightarrow
R{g_C}_\ast \La \longrightarrow.$
Applying the functor ${j}_!(\cK_U \otimes \delta_U^\ast (-))$ to this triangle, we obtain a distinguished triangle
$${j}_!({\cal{K}}_{U} \otimes \delta_{U}^\ast c_\ast \La)(-d)[-2d] \longrightarrow
{j}_!{\cal{K}}_{U}  \longrightarrow {j}_!({\cal{K}}_{U}
\otimes \delta_{U}^\ast R{g_C}_\ast \La) \longrightarrow.$$ 
The canonical isomorphism on the right term ${j}_!{\cal{K}}_{U}
\otimes \delta_{X}^\ast R{g_{\bar{C}}}_\ast \La \simeq
{j}_!({\cal{K}}_{U} \otimes \delta_{U}^\ast R{g_C}_\ast \La)$
and the isomorphism $\La_U(d)[2d] \simeq \cK_U$
induce a
distinguished triangle
$${i_1}_\ast {j'}_!\La_{U \cap C} \longrightarrow
{j}_!{\cal{K}}_{U}  \longrightarrow
{j}_!{\cal{K}}_{U} \otimes \delta_{X}^\ast R{g_{\bar{C}}}_\ast \La \longrightarrow$$
where ${j'}:C \cap U \longrightarrow  {\bar{C}} \cap X $ is the open immersion and 
$i_1:{\bar{C}} \cap X \longrightarrow X$ is the closed immersion and hence
 a long exact sequence 
\begin{align}\label{exact}
H_{\bar{C} \cap S}^0({\bar{C}}\cap X,{j'}_!\La_{U \cap C}) \longrightarrow
 H_{\bar{C} \cap S}^0(X, {j}_!{\cK}_U) \longrightarrow
 H_{\bar{C} \cap S}^0(X, \juq \cK_U \otimes \delta_X^\ast R{g_{\bar{C}}}_\ast \La) \longrightarrow \cdots.
\end{align}

\begin{lemma}
\label{lem:1}
Let the notation be as above. Furhter we assume that $X$ is smooth over $k.$
If $C=\delta_U(U)$ is the diagonal,
the difference
$C_{S,!}^0({j}_! \cF,j_!u)-{\rm Tr}(u) \cdot
C_{S,!}^0({j}_! \La_U) \in H_{S}^0(X, \juq \cK_U \otimes \delta_X^\ast {Rg_X}_\ast \La)$
is in the image of the injection
$H_{S}^0(X, {j}_!{\cK}_U)
\longrightarrow
H_{S}^0(X, \juq \cK_U \otimes \delta_X^\ast {Rg_X}_\ast \La)$
 where ${\rm Tr}(u)$ denotes the image of $u \in H_U^0(U \times U,\cH_U)\simeq {\rm End}_U(\cF)$ under the trace map ${\rm End}_U(\cF) \longrightarrow \La.$
\end{lemma}
\begin{proof}This is proved in the same way as in \cite[Section 5, Lemma 5.2.4.1]{AS}.
\end{proof}
\begin{definition}\label{aho1}
Let the notation and the assumption be as in Lemma \ref{lem:1}.
We call the element
in $H_S^0(X,{j}_!\cK_U)$ lifting the difference
$C_{S,!}^0({j}_! \cF,j_!u)-{\rm Tr}(u) \cdot
C_{S,!}^0({j}_! \La_U) \in H_{S}^0(X, \juq \cK_U \otimes \delta_X^\ast {Rg_X}_\ast \La)$ by Lemma \ref{lem:1}
{\it the refined
localized characteristic class of ${j}_!\cF$ }
and denote it by $C_{S,!}^{00}({j}_! \cF,{j}_!u).$
We call the image
of the element $C_{S,!}^{00}({j}_!
\cF,{j}_!u)$ under the canonical map $H_{S}^0(X, {j}_!\\{\cK}_U)
\longrightarrow H_{S}^0(X,\cK_X)$ {\it the localized characteristic
class of ${j}_!\cF$} and denote it by $C_{S}^{00}({j}_! \cF,{j}_!u)$.
\end{definition}
\begin{remark}\label{aho2}
In the case where $X$
is smooth and $C=\delta_U(U),$
by the exact sequence similar 
to (\ref{exact}) and the purity theorem, we have an isomorphism
$H_S^0(X,\cK_X) \simeq H_S^0(X,\cK_X \otimes \delta_X^\ast {Rg_X}_\ast \La).$
Hence we obtain the localized characteristic class 
$C_S^0({j}_!\cF)$ in $H_S^0(X,\cK_X)$
without taking the difference. (c.f.\ \cite[Definition 5.2.1]{AS}.)
The class 
$C_{S,!}^{00}({j}_! \cF,j_!u) \in H_S^0(X,j_!\cK_U)$
in Definition \ref{aho1} goes to the difference
$C_S^0(j_!\cF,j_!u)-{\rm Tr}(u) \cdot C_S^0(j_!\La) \in H_S^0(X,\cK_X)$
by the canonical map $H_S^0(X,j_!\cK_U)
\longrightarrow H_S^0(X,\cK_X).$
\end{remark}

\subsection{Logarithmic localized characteristic class}
In this subsection, for a $\La$-sheaf, we will define a cohomology class with support
on its wild locus by killing its tame ramification,
which we call the logarithmic localized characteristic class. 
We defined the localized characteristic class as a cohomology class
with support on the boundary locus in subsection 3.1.
To kill the tame ramification, we use logarithmic blow-up.
For a smooth sheaf of rank 1, we introduce a more elementary definition of the logarithmic localized characteristic class
 in \cite[Definition 2.4]{T}.

Let $X$ be a smooth scheme of dimension $d$ over $k,$
$U \subset X$ an open subscheme. 
We assume that the complement $X \backslash U=\bigcup_{i \in I}D_i$
is a divisor with simple normal crossings.
Let
$$(X \times X) \widetilde{} \subset (X \times X)'$$
denote the log product and the log blow-up
with respect to divisors $\{D_i\}_{i \in I}$
defined in \cite[subsection 2.2]{AS}
and \cite[subsection 1.1]{KS}.
\begin{lemma}\label{geom}
Let the notation be as above.
We consider the following cartesian diagram
\[\xymatrix{
X' \ar[r]^{\!\!\!\!\!\!\!\!\!\!\!i'}\ar[d] &
(X \times X)'  \ar[d]^{f} \\
X \ar[r]^{\!\!\!\!\!\!\delta} &
X \times X
}
\]where $f:(X \times X)' \longrightarrow X \times X$
is the projection and $\delta:X \longrightarrow X \times X$
is the diagonal closed immersion.
Then $X'$ is the union of the diagonal $X$
and $(\mathbb{P}^1)^{\sharp{J}}$-bundles
for a subset $\phi \neq J \subset I$
over $D_J$
where
$D_J$
is the intersection of $\{D_i\}_{i \in J}$ in $X.$
\end{lemma}
\begin{proof}
For $i \in I,$
we define $(X \times X)'_i$
to be the blow up of $X \times X$
along the closed subscheme
$D_i \times D_i \subset X \times X.$
Let $X'_i$
denote the inverse image of the diagonal $X$
by the projection $(X \times X)'_i 
\longrightarrow X \times X$ for $i \in I.$
By the definition of the log blow-up,
$X'_i$ is the union of the diagonal 
$X \subset (X \times X)'_i$
and a $\mathbb{P}^1$-bundle over $D_i$.
Since $(X \times X)'$
is the fiber product of $(X \times X)'_i$
($i \in I$)
over $X \times X$, $X'$
is the fiber product of the schemes $\{X'_i\}_{i \in I}$
($i \in I$) over $X.$
Therefore $X'$
is the union of the diagonal $X$ and $(\mathbb{P}^1)^{\sharp{J}}$-bundles
for $\phi \neq J \subset I$
over $D_J$
where
$D_J$
is the intersection of $\{D_i\}_{i \in J}$ in $X.$
Hence the assertion follows.
\end{proof}

We consider the following situation. Let $X$ be a scheme of
dimension $d$ over $k.$ Let $U
\subset V$ be
 open subschemes of $X,$ and $S=X \backslash V$
and $D \cup S=X \backslash U$ the complements respectively.
We assume that $V$ is smooth over $k,$
$D$ is a Cartier divisor and $D \cap V \subset V$ is 
a divisor with simple normal crossings. 
Let $j:U \longrightarrow X,$ $j_V:V \longrightarrow X$
and $j':U \longrightarrow V$
denote the open immersions. Let $U'$ be the complement
of $D:=\bigcup_{i \in I}D_i$ in $X.$
We have $U=V \cap U'.$
Let 
$(X \times X) \widetilde{} \subset (X \times X)'$
and $(V \times V) \widetilde{} \subset (V \times V)'$
denote the log products and the log blow-ups with respect to
$\{D_i\}_{i \in I}$ and $\{D_i \cap V\}_{i \in I}$
respectively.

We
consider the following commutative diagram
\[\xymatrix{
(V \times V)' \ar[d]_{f_V} & (U \times V)' \ar[l]_{\widetilde{j_1}} &
(V \times V) \widetilde{} \ar[l]_{\widetilde{k}_1} \\
V \times V & U \times V \ar[l]\ar[u] &
U \times U \ar[l]\ar[u]_{\widetilde{j}} 
\\}
\]
where $(U \times V)' \subset (V \times V)'$ is the open subscheme
which is the complement of the union of the proper transforms
of $(D_i \cap V) \times V$
 for all $i \in I$.(c.f.\ \cite[Section 2.2]{AS}.)

We consider the cartesian diagram
\begin{align}\label{f1}
\xymatrix{ (X \times X)' \ar[d]_{f} & (V \times X)'
\ar[l]_{\bar{j'_1}}\ar[d]^{f_1} &
(V \times V)' \ar[l]_{\bar{k'_1}}\ar[d]^{f_V} \\
X \times X & V \times X \ar[l]_{j_1} &
V \times V \ar[l]_{k_1}
\\}
\end{align}
where the horizontal arrows are the open immersions
and the vertical arrows are the projections.

Let $\cF$ be a smooth $\La$-sheaf on $U$ which is tamely
ramified along $V \backslash U=V \cap D$. 
We put $\cH_0:=\mathcal{H}om({\rm pr}_2^
\ast \cF,{\rm pr}_1^\ast \cF)$ on $U \times U$ and
$\bar{\cH}:=R\mathcal{H}om({\rm pr}_2 ^\ast j_!\cF,R{\rm pr}_1 ^! j_!\cF)$ on $X
\times X$ respectively.

We recall a construction defined in loc.\ cit. 
We put $\bar{\cH}_V:=R\mathcal{H}om({\rm pr}_2^\ast j'_!\cF,R{\rm pr}_1^!j'_!\cF)$
on $V \times V,$
$\widetilde{\cH}_V:=({\widetilde{j}}_\ast \cH_0)(d)[2d]$
on $(V \times V) \widetilde{}$ and
$\bar{\cH}'_{V}:={\widetilde{j_1}}_! {R\widetilde{k}_1}_\ast
({\widetilde{j}}_\ast \cH_0)(d)[2d]$ 
 on $(V \times V)'$ respectively.
There exists a unique
map 
\begin{align}\label{pul}
f_{V}^\ast \bar{\mathcal{H}}_V
\longrightarrow
\bar{\cH}'_{V}
\end{align}
inducing the canonical isomorphism
$R\mathcal{H}om({\rm pr}_2^\ast \cF,R{\rm pr}_1^!\cF)
\longrightarrow \cH_0(d)[2d]$ on $U \times U$
by \cite[the proof of Proposition 3.1.1.1]{S}.

We put $\bar{\cH'}:=\bar{j'_1}_! R\bar{k'_1}_\ast \bar{\cH}'_{V}$ on
$(X \times X)'.$
We define a map
\begin{equation}\label{11}
f^\ast \bar{\mathcal{H}} \longrightarrow \bar{\cH'}
\end{equation}
to be the composite of the following maps
$$f^\ast \bar{\mathcal{H}} \simeq
f^\ast {j_1}_!{R{k_1}}_\ast \bar{\cH}_V \simeq
{\bar{j'_1}}_! f_1^\ast {R{k_1}}_\ast \bar{\cH}_V
\longrightarrow
{\bar{j'_1}}_!{R{\bar{k'_1}}}_\ast f_{V}^\ast \bar{\mathcal{H}}_V
\longrightarrow
{\bar{j'_1}}_!{R{\bar{k'_1}}}_\ast \bar{\cH}'_{V}=\bar{\cH'}$$
where the first map
is induced by the Kunneth formula
and the second and third maps are induced by
the base change maps
$f^\ast {j_1}_! \simeq {\bar{j'_1}}_!f_1^\ast $ and 
$f_1^\ast {R{k_1}}_\ast \longrightarrow {R{\bar{k'_1}}}_\ast f_{V}^\ast $
and the fourth map is induced by the map (\ref{pul}).
\begin{lemma}\label{5}
Let the notation be as above. Then the adjunction of the map (\ref{11})
$\bar{\cH} \longrightarrow Rf_\ast \bar{\cH'}$ is an isomorphism.
\end{lemma}
\begin{proof}
By \cite[lemma 2.2.4]{AS}, the adjunction $
\bar{\cH}_V \longrightarrow
R{f_V}_\ast \bar{\cH}'_{V}$ is an isomorphism. Since $f$ is proper, the assertion follows
from the cartesian diagram (\ref{f1}) and the definition of $\bar{\cH'}$.
\end{proof}

We consider the following cartesian diagram
\[\xymatrix{
X' \ar[r]^{\!\!\!\!\!\!\!\!i'}\ar[d]^{f'} &
(X \times X)' \ar[d]^{f} & (X \times X)' \backslash X' \ar[l]_{g'}\ar[d]\\
X \ar[r]^{\!\!\!\!\!\!\!\!\delta_X} & X \times X & X \times X \backslash \delta_X(X) \ar[l]_{\!\!\!\!\!\!\!g_X}
\\}
\]where $f:(X \times X)' \longrightarrow X \times X$
is the projection,
$i':X'
\longrightarrow (X \times X)'$ is the closed immersion and
$g':(X \times X)' \backslash X' \longrightarrow (X
\times X)'$ is the open immersion. 
Let
$\delta':X \longrightarrow (X \times X)',$ 
$\widetilde{\delta}_{V}:V \longrightarrow (V \times V) \widetilde{}$
and $\delta'_V:V \longrightarrow (V \times V)'$
 be the
logarithmic diagonal closed immersions induced by the universality of blow-up. 
Let $V'$ denote the intersection $X' \cap (V \times V)'$ in $(X \times X)'.$
Let $g'_V:(V \times V)' \backslash V' \longrightarrow (V \times V)'$
be the open immersion.

We define an evaluation map. The composite of the canonical isomorphism
${\delta'}^\ast
\bar{\cH'} \simeq {j_V}_! {\widetilde{\delta}_{V}}^\ast
(\widetilde{j}_\ast \cH_0)(d)[2d]$ and an evaluation map
${j_V}_!e:{j_V}_!{\widetilde{\delta}_{V}}^\ast (\widetilde{j}_\ast
\cH_0)(d)[2d] \longrightarrow {j_V}_!\La_V(d)[2d]={j_V}_!
\cK_V$(\cite[(2.9)]{AS}) induces 
an evaluation map
\begin{align}\label{6.2}
e':{\delta'}^\ast \bar{\cH'} \longrightarrow {j_V}_!
\cK_V.
\end{align}
The map (\ref{11}) induces the pull-back
\begin{align}\label{zz}
f^\ast:H_X^0(X \times X,\bar{\cH}) \longrightarrow H_{X'}^0((X
\times X)',\bar{\cH'}).
\end{align}
The canonical map $\La \longrightarrow Rg'_\ast\La$
induces a map
$$H_{X'}^0((X \times X)',\bar{\cH'})
\longrightarrow H_{X'}^0((X \times X)',\bar{\cH'}
\otimes Rg'_\ast\La).$$

\begin{lemma}\label{7}
Let the notation be as above. Then the canonical map
$$H_{X' \backslash V'}^0((X \times X)',\bar{\cH'}
\otimes Rg'_\ast\La) \longrightarrow H_{X'}^0((X \times
X)',\bar{\cH'} \otimes Rg'_\ast\La)$$ is injective.
\end{lemma}
\begin{proof}
By the localization sequence, it suffices to prove
$H_{V'}^{-1}((V \times V)',{\bar{\cH'}}_V \otimes {R{g'_V}}_\ast \La)=0.$
In the following, we may assume that $V=X$ and
$\mathcal{F}$ is tamely ramified along the boundary $X \backslash U=D$
which is a divisor with simple normal crossings.
We will prove the vanishing $H_{X'}^{-1}((X \times X)',\bar{\cH'} \otimes Rg'_\ast\La)=0.$
Let $P$ denote the closed subscheme $X' \backslash U$ of $(X \times X)'.$
By a similar argument to the proof of Lemma \ref{lll1},
we obtain an isomorphism
$H_{P}^{-1}((X \times X)',\bar{\cH'} \otimes Rg'_\ast\La)
\simeq 
H_{X'}^{-1}((X \times X)',{\bar{\cH'}} \otimes Rg'_\ast\La).$
We have a distinguished triangle
$${i'}_\ast ({i'}^\ast {\bar{\cH'}} \otimes {R{i'}}^!\La)
\longrightarrow
{\bar{\cH'}}
\longrightarrow
{\bar{\cH'}} \otimes Rg'_\ast\La
\longrightarrow
$$
and hence a long exact sequence
$$H_{P}^{-1}((X \times X)',{\bar{\cH'}})
\longrightarrow
H_{P}^{-1}((X \times X)',{\bar{\cH'}} \otimes Rg'_\ast\La)
\longrightarrow
H_{P}^0(X',{i'}^\ast {\bar{\cH'}} \otimes {R{i'}}^!\La)
\longrightarrow
$$
$H_{P}^{0}((X \times X)',{\bar{\cH'}}) \longrightarrow \cdots.$

Because we have
$H_{P}^{i}((X \times X)',{\bar{\cH'}}) \simeq
H_{X \backslash U}^{i}(X \times X,{\bar{\cH}}) \simeq 0$
for all $i$ by the isomorphism ${R{f}}_\ast {\bar{\cH'}} \simeq {\bar{\cH}}$ proved in \cite[Lemma 2.2.4]{AS},
we obtain an isomorphism
$H_{P}^{-1}((X \times X)',{\bar{\cH'}} \otimes Rg'_\ast\La)
\simeq
H_{P}^0(X',{i'}^\ast {\bar{\cH'}} \otimes {R{i'}}^!\La)$
by the above long exact sequence.

Since we have $X'={\delta'}(X) \cup P$
and $D={\delta'}(X) \cap P,$ 
we acquire the following long exact sequence by the excision
{\small
$$H^0(D,{i'}_{D}^\ast {\bar{\cH'}} \otimes R{i'}_{D}^!\La)
\longrightarrow
H_D^0(X,{\delta'}^\ast {\bar{\cH'}} \otimes {R{\delta'}}^!\La)
\oplus
H^0(P, i_{P}^\ast {\bar{\cH'}} \otimes Ri_{P}^!\La)
\longrightarrow
H^0_P(X',{i'}^\ast {\bar{\cH'}} \otimes {R{i'}}^!\La)
$$}
$\longrightarrow H^1(D,{i'}_{D}^\ast {\bar{\cH'}} \otimes R{i'}_{D}^!\La)
\longrightarrow \cdots$
where $i_P:P \longrightarrow (X \times X)'$
and $i'_D:D \longrightarrow (X \times X)'$
are the closed immersions.
By the purity theorem and \cite[Corollary 2.21(3)]{S}, we obtain isomorphisms
$H_D^0(X,{\delta'}^\ast {\bar{\cH'}} \otimes {R{\delta'}}^!\La)
\simeq H_D^0(X,\widetilde{\delta} ^\ast \widetilde{j}_\ast \cH_0)
\simeq
H_D^0(X,j_\ast \mathcal{E}nd(\cF))=0.$
Let $i_D:D \longrightarrow X$ be the closed immersion.
Since we have an isomorphism
$H^i(D,{i'}_{D}^\ast {\bar{\cH'}} \otimes R{i'}_{D}^!\La)
\simeq
H^i(D, j_\ast \mathcal{E}nd(\cF)|_D \otimes R{i}_{D}^!\La),$ we acquire $H^i(D,{i'}_{D}^\ast {\bar{\cH'}} \otimes R{i'}_{D}^!\La)=0$
for $i \leq 1$ again by the purity theorem.
Therefore we obtain an isomorphism
$H^0(P, i_{P}^\ast {\bar{\cH'}} \otimes Ri_{P}^!\La)
\simeq
H^0_P(X',{i'}^\ast {\bar{\cH'}} \otimes {R{i'}}^!\La).$

By lemma \ref{geom}, the purity theorem and the excision,
it suffices to show that 
${Rf_J}_\ast ({\bar{\cH'}}|_{P_J})$
is acyclic for each $\phi \neq J \subset I$
where $P_J$ is a $(\mathbb{P}^1)^{\sharp{J}}$-bundle
over $D_J$ and $f_J:P_J \longrightarrow D_J$ is the projection.
Since the assertion is \'{e}tale local,
we may assume that
$\cF$ is the tensor product $\bigotimes_{i \in I}\cF_i$
where $\cF_i$ is the extension by zero of a 
smooth sheaf on the complement $X \backslash D_i$ for
$i \in I$ in the same way as in \cite[the proof of Lemma 2.2.4]{AS}.
Since $(X \times X)'$ is the fiber product
of $(X \times X)'_i$ over $X \times X$
and $P_J$ is the fiber product of $P_i \subset (X \times X)'_i$
for $i \in J$ and the diagonal $X \subset (X \times X)'_i$
for $i \in I \backslash J$ over $X$
where $P_i$ is the $\mathbb{P}^1$-bundle over $D_i.$
Hence, by the Kunneth formula,
it is reduced to the case where $D$ is a smooth divisor.
By the cartesian diagram
\[\xymatrix{
P \ar[r]^{\!\!\!\!\!\!\!\!i_P}\ar[d]^{f_P}& (X \times X)' \ar[d]^{f} \\
D \ar[r]^{\!\!\!\!\!i'_D}& X \times X,
}
\]we obtain
${R{f_P}}_\ast (\bar{\cH'})|_P \simeq
({R{f}}_\ast \bar{\cH'})|_{D}$
by the proper base change theorem.
By the isomorphism ${R{f}}_\ast {\bar{\cH'}} \simeq {\bar{\cH}}$ proved in loc.\ cit., the assertion follows.
\end{proof}
We have a cohomological correspondence $j_!1={ \rm id}_{j_!\cF}$ in
$H_X^0(X \times X,\bar{\cH})$ and 
its pull-back $f^\ast {\rm id}_{j_!\cF} \in H_{X'}^0((X \times
X)',\bar{\cH'})$ by (\ref{zz}).
\begin{lemma}\label{existence}
Let the notation be as above. There exists 
a unique element $(f^\ast {\rm id}_{j_!\cF})'$
in $H_{X' \backslash V'}^0((X \times
X)',\bar{\cH'} \otimes Rg'_\ast\La)$
which is sent to $f^\ast {\rm id}_{j_!\cF}$
 by the canonical map
$H_{X' \backslash V'}^0((X \times
X)',\bar{\cH'} \otimes Rg'_\ast\La)
\longrightarrow
H_{X'}^0((X \times
X)',\bar{\cH'} \otimes Rg'_\ast\La).$
\end{lemma}
\begin{proof}By the localization sequence and Lemma \ref{7},
it suffices to show that
the element $f^\ast {\rm id}_{j_!\cF}$ goes to
zero under the restriction map $H_{X'}^0((X \times
X)',\bar{\cH'} \otimes Rg'_\ast\La) \longrightarrow
H_{V'}^0((V \times V)',\bar{\cH'}_V \otimes R{g'_V}_\ast \La).$
Namely we prove the following vanishing
${f_V}^\ast {\rm id}_{j'_!\cF}=0$
in $H_{V'}^0((V \times V)',\bar{\cH'}_V \otimes R{g'_V}_\ast \La).$
Since $\cF$ is tamely ramified along $V \backslash U,$
we have an equality ${f_V}^\ast {\rm id}_{j'_!\cF}=e \cup [V]$
in $H_{V'}^0((V \times V)',\bar{\cH'}_V)$
by \cite[Proposition 3.1.1.2]{S}.
We consider the following commutative diagram
\[\xymatrix{
e \cup [V] \in H_V^0((V \times V) \widetilde{},\widetilde{\mathcal{H}}_V) \ar[r] &
H_{V}^0((V \times V) \widetilde{},\widetilde{\mathcal{H}}_V \otimes {R\widetilde{g}_V}_\ast \La)\\
[V] \in H_{V}^{2d}((V \times V) \widetilde{},\La(d)) \ar[r]\ar[u]^{e \cup} &
H_V^{2d}((V \times V) \widetilde{},{R\widetilde{g}_V}_\ast \La(d)) \simeq 0 \ar[u]^{e \cup}\\
}
\]where $\widetilde{g}_V:(V \times V) \widetilde{} \backslash \widetilde{V}
\longrightarrow (V \times V) \widetilde{}$ is the open immersion and 
the horizontal arrows are induced by the canonical map $\La \longrightarrow {R\widetilde{g}_V}_\ast \La.$
Since we have $H_{V}^{2d}((V \times V) \widetilde{},{R\widetilde{g}_V}_\ast \La(d))=0,$
we acquire an equality $e \cup [V]=0$ in $H_{V}^0((V \times V) \widetilde{},\widetilde{\cH}_V
\otimes {R\widetilde{g}_V}_\ast \La) \simeq H_V^0((V \times V)',\bar{\cH'}_V \otimes {Rg'_V}_\ast \La)$ by the above commutative diagram.
Hence we obtain the vanishing ${f_V}^\ast {\rm id}_{j'_!\cF}=0$
in $H_{V'}^0((V \times V)',\bar{\cH'}_V \otimes R{g'_V}_\ast \La).$
\end{proof}

\begin{lemma}\label{ll}
Let the notation be as above. 
Further we assume that $X$ is smooth over $k.$
Then the canonical map induced by the map
$\La \longrightarrow {\delta'}^\ast Rg'_\ast
\La$
$$H_{S}^0(X,\cK_X) \longrightarrow
H_{S}^0(X,\cK_X \otimes {\delta'}^\ast Rg'_\ast
\La)$$ is an isomorphism.
\end{lemma}
\begin{proof}
This is proved in \cite[Lemma 2.3]{T}.
\end{proof}

The pull-back by $\delta'$ and the evaluation map (\ref{6.2})
induce a map 
$$e' \cdot {\delta'}^\ast:H_{X' \backslash V'}^0((X \times X)',\bar{\cH'} \otimes Rg'_\ast\La)
\longrightarrow H_{S}^0(X,{j_V}_!\cK_V \otimes
{\delta'}^\ast Rg'_\ast\La).$$ 
We have obtained the following maps
\[\xymatrix{
H_{X'}^0((X \times X)',\bar{\cH'}) \ar[r]^{\!\!\!\!\!\!\!\!\!\!\!\!\!\!\rm can.}&
H_{X'}^0((X \times X)',\bar{\cH'} \otimes Rg'_\ast\La) & \\
{\rm id}_{j_! \cF} \in H_X^0(X \times X,\bar{\cH}) \ar[u]^{f^\ast }_{(\ref{zz})} &H_{X' \backslash V}^0((X \times X)',\bar{\cH'} \otimes Rg'_\ast\La) \ar[u]^{\rm inj.}_{{\rm Lemma} \ref{7}}\ar[r]^{e' \cdot {\delta'}^\ast\!\!\!\!\!} &
H_{S}^0(X,{j_V}_!\cK_V \otimes
{\delta'}^\ast Rg'_\ast\La).
}
\]By the map $e' \cdot {\delta'}^\ast$ and Lemmas
\ref{7} and \ref{existence}, we obtain a class $e' \cdot {\delta'}^\ast (f^\ast {\rm
id}_{j_! \cF})'$ in $H_{S}^0(X,{j_V}_!\cK_V \otimes
{\delta'}^\ast Rg'_\ast\La).$ We put $C_{S,!}^{{\rm
log},0}(j_!\cF):=e' \cdot {\delta'}^\ast (f^\ast {\rm id}_{j_! \cF})'.$
We denote by $C_{S}^{{\rm log},0}(j_!\cF) \in
H_{S}^0(X,\cK_X)$ the image of the class $C_{S,!}^{{\rm
log},0}(j_!\cF)$ under the canonical map
$H_{S}^0(X,{j_V}_!\cK_V \otimes {\delta'}^\ast
Rg'_\ast\La) \longrightarrow H_{S}^0(X,\cK_X \otimes
{\delta'}^\ast Rg'_\ast\La) \simeq H_{S}^0(X,\cK_X).$
(by Lemma \ref{ll}.) 

\begin{definition}\label{bakayarou}
Let the notation and the assumption be as in Lemma \ref{ll}.
We call the class $C_S^{\text{log},0}(j_!\cF)$
in $H_S^0(X,\cK_X)$
{\it the logarithmic localized characteristic class of $j_!\cF.$}
We call the element $C_{S,!}^{\text{log},0}(j_!\cF)$
in $H_S^0(X,{j_V}_!\cK_V \otimes {\delta'}^\ast Rg'_\ast\La)$
{\it the refined logarithmic localized characteristic class
of $j_!\cF.$}
We put the difference
$C_S^{\text{log},00}(j_!\cF):=C_{S}^{\text{log},0}(j_!\cF)-\text{rk}(\cF) \cdot C_{S}^{\text{log},0}(j_!\La)
\in H_S^0(X,\cK_X).$
\end{definition}

Let the notation be as above.
In the following, we assume that
$V=X,$ $X \backslash U=D$ is a divisor with simple normal crossings 
and that $\cF$ is tamely ramified along the boundary $D.$
We will prove the vanishing of the localized
characteristic class, i.e.\ $C_D^{00}(j_!\cF)=0.$
This vanishing plays a key role in the proof
of the localized Abbes-Saito formula.
\begin{remark}
Let the notation be as in Definition \ref{bakayarou}.
We expect that the logarithmic localized characteristic class
$C_S^{\text{log},00}(j_!\cF)$
is sent to the localized characteristic class
$C_{S \cup D}^{00}(j_!\cF)$ by the canonical
map $H_S^0(X,\cK_X) \longrightarrow H_{S \cup D}^0(X,\cK_X).$
If we admit this,
the vanishing $C_D^{00}(j_!\cF)=0$
will follow from Definition \ref{bakayarou}
by putting
$X=V, S=\emptyset.$
However we do not know a proof.
We give a proof of the vanishing $C_D^{00}(j_!\cF)=0$
in the following.
\end{remark}
We write $\bar{\cH'}_{\cF}$ and $e_{\cF} \in \Gamma(X,{\widetilde{j}}_\ast \cH_0|_X)$
for the sheaf $\bar{\cH'}$ and the unique section
$e \in \Gamma(X,{\widetilde{j}}_\ast \cH_0|_X)$ lifting the identity
$\Gamma(U,\cH_0|_U)$ to emphasize that they are associated to
the sheaf $\cF.$

The canonical map $\La \longrightarrow f^\ast
{R{g_{X}}}_\ast \La$ induces a map $H_{X'}^0((X \times X)',
\bar{\cH'}_{\cF}) \longrightarrow H_{X'}^0((X \times
X)', \bar{\cH'}_{\cF} \otimes f^\ast {R{g_{X}}}_\ast \La).$
By the same argument as the proof of Lemma \ref{lll1},
the canonical map $H_{X' \backslash U}^0((X \times
X)', \bar{\cH'}_{\cF} \otimes f^\ast {R{g_{X}}}_\ast
\La) \longrightarrow
H_{X'}^0((X
\times X)', \bar{\cH'}_{\cF} \otimes f^\ast
{R{g_{X}}}_\ast \La)$ is an isomorphism.
  The image of $f^\ast {\rm id}_{{j}_!\cF}$
under the composite
$H_{X'}^0((X \times X)',
\bar{\cH'}_{\cF}) \longrightarrow H_{X'}^0((X \times
X)', \bar{\cH'}_{\cF} \otimes f^\ast {R{g_{X}}}_\ast \La)
\simeq
H_{X' \backslash U}^0((X \times
X)', \bar{\cH'}_{\cF} \otimes f^\ast {R{g_{X}}}_\ast
\La)$ denotes $(f^\ast {\rm id}_{{j}_!\cF})^{\rm loc}$.
The pull-back by ${\delta}'$ and the evaluation map
$e'$(\ref{6.2}) induce $e' \cdot {\delta'}^\ast:H_{X' \backslash U}^0((X \times X)',
\bar{\cH'}_{\cF} \otimes f^\ast {R{g_{X}}}_\ast \La)
\longrightarrow H_{D}^0(X,\cK_{X} \otimes {\delta}^\ast
{R{g_{X}}}_\ast \La)\simeq H_{D}^0(X,\cK_{X}).$(by Remark \ref{aho2}.) The image of the
element $(f^\ast{\rm id}_{{j}_!\cF})^{\text{loc}} \in H_{X' \backslash U}^0((X \times
X)', \bar{\cH'}_{\cF} \otimes f^\ast {Rg_X}_\ast \La)$ under this map
denotes $e' \cdot {\delta'}^\ast (f^\ast {\rm
id}_{{j}_!\cF})^{\rm loc} \in H_{D}^0(X,\cK_{X}).$

\begin{lemma}\label{GOS-lemma4}
Let the notation be as above.
Then we have the following vanishing
$$C_{D}^{00}({j}_!\cF)=0$$
in $H_{D}^0(X,\cK_X).$
\end{lemma}
\begin{proof}
We prove an equality
$C_{D}^0({j}_!\cF)=e' \cdot {\delta'}^\ast (f^\ast {\rm
id}_{{j}_!\cF})^{\rm loc}$ in $H_{D}^0(X,\cK_{X}).$
This follows from Definition \ref{aho1}, Remark \ref{aho2} and the following commutative diagram
\[\xymatrix{
{\rm
id}_{{j}_!\cF} \in H_{X}^0(X \times X,{\bar{\cH}}_{\cF}) \ar[r]\ar[d]^{f^\ast} &
H_{X}^0(X \times X,{\bar{\cH}}_{\cF} \otimes {R{g_{X}}}_\ast \La)
\ar[d]^{f^\ast}
\\
f^\ast {\rm
id}_{{j}_!\cF} \in H_{X'}^0((X \times X)', \bar{\cH'}_{\cF}) \ar[r] &
H_{X'}^0((X \times X)', \bar{\cH'}_{\cF} \otimes
f^\ast {R{g_{X}}}_\ast \La) }
\]
\[\xymatrix{
 & H_{D}^0(X \times X,{\bar{\cH}}_{\cF} \otimes {R{g_{X}}}_\ast
\La) \ar[d]^{f^\ast}\ar[r]^{e \cdot {\delta}^\ast
\!\!\!\!\!\!}\ar[l]^{\!\!\!\!\!\!\!\!\!\!\!\!\!\!\!\!\!\!\!\!\!\!\!\!\!\!\!\!\!\!\!\!\!\!\!\!
 \rm can.}_{\!\!\!\!\!\!\!\!\!\!\!\!\!\!\!\!\!\!\!\!\!\!\!\!\!\!\!\!\!\!\!\!\!\!\!\! \simeq} & H_{D}^0(X,\cK_{X} \otimes {\delta}^\ast {R{g_{X}}}_\ast \La)\ar[d]^{\rm id} \\
 & H_{X' \backslash U}^0((X \times X)',
\bar{\cH'}_{\cF} \otimes f^\ast {R{g_{X}}}_\ast \La)
\ar[r]^{\!\!\!\!\!\!\!e' \cdot {\delta'}^\ast
}\ar[l]^{\!\!\!\!\!\!\!\!\!\!\!\!\!\!\!\!\!\!\!\!\!\!\!\!\!\!\!\!\!\!\!\!\!\!\!\!\!\!\!\!\!\!\!\!\!\!\!\!\!\!\!\!
\rm
can.}_{\!\!\!\!\!\!\!\!\!\!\!\!\!\!\!\!\!\!\!\!\!\!\!\!\!\!\!\!\!\!\!\!\!\!\!\!\!\!\!\!\!\!\!\!\!\!\!\!\!\!\!\!\!\!
\simeq}&
H_{D}^0(X,\cK_{X} \otimes {\delta}^\ast {R{g_{X}}}_\ast \La) \simeq H_D^0(X,\cK_X)\\
}
\]where the vertical arrows are induced by the map (\ref{11}) $f^\ast \bar{\cH} \longrightarrow \bar{\cH'}_{\cF}.$ 
By this diagram and $f^\ast {\rm id}_{{j}_!\cF}=e_{\cF} \cup [X]$ by \cite[Proposition 3.1.1.2]{S}, we acquire equalities
$C_{D}^0({j}_!\cF)=e' \cdot {\delta'}^\ast
(f^\ast {\rm id}_{{j}_!\cF})^{\rm loc}=e'(e_{\cF})
\cdot {\delta'}^\ast [X]={\rm rk}(\cF) \cdot
{\delta'}^\ast [X]$ and $C_{D}^0({j}_!\La_U)=e'(e_{\La_U}) \cdot {\delta'}^\ast [X]={\delta'}^\ast
[X]$. Hence the assertion follows.
\end{proof}

In the following, we calculate the localized characteristic class by the localized Chern class
using Lemma \ref{GOS-lemma4} in tamely ramified case. We will not use the results in the following sections.
We recall the definition of the localized Chern
class from \cite[Section 3.4]{KS}. Let $X$ be a scheme of finite type over $k$
and $Z \subset X$ be a closed subscheme. Let $\mathcal{E}$ and $\cF$
be locally free $\mathcal{O}_X$-modules of rank $d$ and
$f:\mathcal{E} \longrightarrow \cF$ be an $\mathcal{O}_X$-linear
map. We assume that $f:\mathcal{E} \longrightarrow \cF$ is an
isomorphism on $X \backslash Z.$ We consider the complex
$\mathcal{K}=[\mathcal{E} \longrightarrow \cF]$ of
$\mathcal{O}_X$-modules by putting $\cF$ on degree 0. Then, the
localized Chern class $c_Z^X(\mathcal{K})-1$ is defined as an
element of $CH^\ast(Z \longrightarrow X)$ in \cite[Chapter
18.1]{Fu}. We define an element
$c(\cF-\mathcal{E})_Z^X=(c_i(\cF-\mathcal{E})_Z^X)_{i>0}$ of
$CH^\ast(Z \longrightarrow X)$ by
$$c(\cF-\mathcal{E})_Z^X=c(\mathcal{E}) \cap (c_Z^X(\mathcal{K})-1).$$
In other words,
we put $c_i(\cF-\mathcal{E})_Z^X=\Sigma_{j=0}^{{\rm {\rm min}}(d,i-1)}c_j(\mathcal{E})
 \cap {c_{i-j}}_Z^X(\mathcal{K})$
for $i>0.$
The image of $c(\cF-\mathcal{E})_Z^X$
in
$CH^\ast(X)$
is the difference
$c(\cF)-c(\mathcal{E})$ of
Chern classes.

\begin{lemma}\label{44}
Let $X$ be a smooth scheme over $k$ of dimension $d$, $U$ an open dense subscheme and
$D$ the complement $X \backslash U$.
We assume that $D$ is a divisor with simple normal crossings of $X$.
Let $j:U \longrightarrow X$
be the open immersion.
Let $\{D_i\}_{i \in I}$ be the irreducible components of $D.$
For a subset $J \subset I$, we put $D_J:=\bigcap_{i \in J} D_i$ and
$B_J:=\bigcup_{i \notin J} (D_J \cap D_i).$ 
Let $j_J:D_J - B_J \longrightarrow D_J$ be the open immersion.\\
1. Then, we have
$$
C_D^0(j_! \La_U)=-\Sigma_{r=1}^{{\rm min}(d,n)} \Sigma_{|J|=r, J \subseteq I} C({j_J}_! \La_{D_J - B_J})$$
in $H_D^0(X,\cK_X)$ where $|I|=n.$\\
2. We have
$$c_d(\Omega^1_{X \slash k}({\rm log} D)-\Omega^1_{X \slash k})_D^X \cap[X]
=-\Sigma_{r=1}^{{\rm min}(d,n)} \Sigma_{|J|=r, J \subseteq I}(-1)^{r}c_{d-r}(\Omega^1_{D_J \slash k}({\rm log} B_J)) \cap [D_J]$$
in $CH_0(D).$
\end{lemma}
\begin{proof}The assertion 1 is easy. We omit a proof.
We prove 2. We have an exact sequence
$$0 \longrightarrow \Omega^1_{X \slash k} \longrightarrow \Omega^1_{X \slash k}({\rm log}D) \longrightarrow \bigoplus_{i \in I} \mathcal{O}_{D_i} \longrightarrow 0.$$ 
We put $\mathcal{K}:=\bigoplus_{i \in I} \mathcal{O}_{D_i}.$
The above sequence induces equalities
$c(\Omega^1_{X \slash k}({\rm log}D)-\Omega^1_{X \slash k})_D^X \cap [X]
=
c(\Omega^1_{X \slash k}) \cap (c_D^X(\mathcal{K})-1) \cap [X]
=c(\Omega^1_{X \slash k}) \cap c_D^X(\mathcal{K}) \cap (1-c_{D}^X(\mathcal{K})^{-1})
\cap [X]$ and
$(1-c_{D}^X(\mathcal{K})^{-1})
\cap [X]=-\Sigma_{r=1}^{{\rm min}(n,d)} \Sigma_{J \subset I, |J|=r} (-1)^r[D_J]. $
Therefore we obtain an equality
$c(\Omega^1_{X \slash k}({\rm log}D)-\Omega^1_{X \slash k})_D^X \cap [X]
=-\Sigma_{r=1}^{{\rm min}(n,d)}(-1)^r \Sigma_{J \subset I, |J|=r} c(\Omega^1_{X \slash k}) \cap c_D^X(\mathcal{K})
\cap [D_J].$
On the other hand, the following equality holds
$c(\Omega_{X \slash k})=c(\Omega^1_{X \slash k}({\rm log}D)) \cap c_d^X(\mathcal{K})^{-1}.$
Hence we acquire
$c(\Omega^1_{X \slash k}({\rm log}D)-\Omega^1_{X \slash k})_D^X \cap [X]
=-\Sigma_{r=1}^{{\rm min}(n,d)}(-1)^r \Sigma_{J \subset I, |J|=r} c(\Omega^1_{X \slash k}({\rm log}D)) \cap [D_J].$
The assertion follows from an equality $c(\Omega^1_{X \slash k}({\rm log}D)) \cap [D_J]
=c(\Omega^1_{D_J \slash k}({\rm log}B_J)) \cap [D_J].$
\end{proof}

\begin{corollary}\label{45}
Let the notation be as in Lemma \ref{44} and $\cF$ be a smooth $\La$-sheaf on $U$
which is tamely ramified along $D$.
Then, we have
$$C_D^0(j_!\cF)=(-1)^d \cdot {\rm rk}(\cF) \cdot c_d(
\Omega^1_{X \slash k}( {\rm log} D)-\Omega^1_{X \slash k})_D^X \cap [X]$$
in $H_{D}^0(X,\cK_X).$
\end{corollary}
\begin{proof}
By Lemma \ref{GOS-lemma4} and the assumption that $\cF$ is tamely ramified
along $D$,
we obtain an equality
$C_D^0(j_! \cF)-{\rm rk}(\cF) \cdot C_D^0(j_! \La_U)=0.$
Therefore the assertion is reduced to an equality
$ C_D^0(j_! \La_U)=(-1)^d \cdot c_d(\Omega^1_{X \slash k}({\rm log} D)-\Omega^1_{X \slash k})_D^X \cap [X].$
By Lemma \ref{44}.1, the following equality holds
$C_D^0({j_!} \La_U)=-\Sigma_{r=1}^{{\rm min}(d,n)} \Sigma_{|J|=r, J \subseteq I} C({j_J}_! \La_{D_J - B_J})$ in $H_D^0(X,\cK_X).$
By \cite[Crollary 2.2.5.1]{AS}, we acquire an equality
$C({j_J}_! \La_{D_J-B_J})=(-1)^{d-r}c_{d-r}(\Omega^1_{D_J \slash k}({\rm log}B_J)) \cap [D_J].$
Hence the assertion follows from Lemma \ref{44}.2.
\end{proof}

\subsection{Pull-back}
\label{section:pullback}
\noindent
In this subsection,
we will prove the compatibility
of the refined localized characteristic class with pull-back.
Let
$X$ and $Y$ be schemes over $k$, $U \subseteq X$ and $V \subseteq Y$ open dense subschemes smooth of dimension $d$ over $k,$ 
and $S:=X \backslash U$ and $T:=Y \backslash V$ the complements
respectively.
We consider a cartesian diagram
\[\xymatrix{
V \ar[r]^{j_V} \ar[d]^{f} & Y \ar[d]^{\bar{f}} & T \ar[l] \ar[d] \\
U \ar[r]^{j} & X & S \ar[l] \\
}
\]where 
$\bar{f}:Y \longrightarrow X$
is a proper morphism and $f:V \longrightarrow U$ is a finite flat morphism.

Let $C \subset U \times U$ be a closed subscheme purely of dimension $d.$ Let $\bar{C}$ be the closure of 
$C$ in $X \times X$, $C' \subset V \times V$ the inverse image of $C \subset U \times U$ by $f \times f:V \times V \longrightarrow
U \times U$ and
$\bar{C}'$ the closure of $C'$ in $Y \times Y$. 
We also assume $C=\bar{C} \cap (X \times U).$ Let $j_C:C \longrightarrow \bar{C}$ and $j_{C'}:C' \longrightarrow
\bar{C'}$ denote the open immersions.
We consider the following cartesian diagram
\[\xymatrix{
Y \times Y \backslash \bar{C'} \ar[r]^{\bar{g'}\!\!\!\!\!}\ar[d] &
Y \times Y \ar[d]^{\bar{f} \times \bar{f}} \\
X \times X \backslash \bar{C} \ar[r]^{\bar{g}\!\!\!\!\!} &
X \times X \\
}
\]where ${\bar{g}:X \times X} \backslash \bar{C} \longrightarrow {X \times X}$ and
$\bar{g}':{Y \times Y} \backslash \bar{C'} \longrightarrow {Y \times Y}$ are the open immersions.

Let $\cal{F}$ be a smooth  $\La$-sheaf on $U$ and $u$ a cohomological correspondence on $C$.
We put $\cF_V$ =$f^\ast \cal{F}$ on $V,$
$\bar{\cH}:=R{\cH}om({\rm pr}_2^\ast {j}_! \cF,R{\rm pr}_1^!{j}_! \cF)$ on $X \times X$ and
$\bar{\cH'}:=R{\cH}om({\rm pr}_2^\ast {j_V}_! \cF_V,R{\rm pr}_1^!{j_V}_! \cF_V)$
on $Y \times Y$ respectively.

We define a map
\begin{align}\label{pull}
{\bar{f}}^\ast ({j}_!\cK_U \otimes \delta_X^\ast R{\bar{g}}_\ast \La)
\longrightarrow
{j_V}_!\cK_V \otimes \delta_Y^\ast R{\bar{g'}}_\ast \La\end{align}
to be the composition of the following maps
$${\bar{f}}^\ast ({j}_!\cK_U \otimes \delta_X^\ast R{\bar{g}}_\ast \La)
\longrightarrow
{\bar{f}}^\ast {j}_!\cK_U \otimes \delta_Y^\ast R{\bar{g'}}_\ast \La
\longrightarrow
{j_V}_!\cK_V \otimes \delta_Y^\ast R{\bar{g'}}_\ast \La$$
where the first map is induced by
the base change map
$(\bar{f} \times \bar{f})^\ast R{\bar{g}}_\ast \La \longrightarrow R{\bar{g'}}_\ast \La$ and the second map is induced by an isomorphism $f^\ast \cK_U \simeq \cK_V$ by the assumption that $U,V$ are smooth schemes of the same dimension.
The map (\ref{pull}) induces the pull-back
\begin{align}\label{pp}
\bar{f}^\ast:H_{\bar{C} \cap S}^0(X,{j}_!\cK_U \otimes \delta_X^\ast R{\bar{g}}_\ast \La)
\longrightarrow
H_{\bar{C'} \cap T}^0(Y,{j_V}_!\cK_V \otimes \delta_Y^\ast R{\bar{g'}}_\ast \La).
\end{align}

\begin{proposition}(Pull-back)\label{29}
Let the notation be as above. 
Then we have an equality
$$ C_{T,!}^0({j_V}_!{\cal{F}}_V,\bar{C}',{j_{C'}}_!(f \times f) ^\ast u)=
\bar{f} ^\ast
C_{S,!}^0(j_!{\cal{F}},\bar{C},{j_C}_!u)$$
in $H_{{\bar{C}'} \cap T}^0(Y,{j_V}_!{\cK_V} \otimes \delta_Y ^\ast R\bar{g}'_\ast \La)$.

\end{proposition}
\begin{proof}
We consider the commutative diagram
\[\xymatrix{
H_{\bar{C'}}^0(Y \times Y,\bar{\cH'}) \ar[r]^{\!\!\!\!\!\!\!\!\!\!\!\!\!\!\!\!\!\!{\rm loc}_{Y,C',\cF_V}}
&
H_{\bar{C'} \backslash C'}^0(Y \times Y,\bar{\cH'} \otimes {R\bar{g'}}_\ast \La) \ar[r]^{\!\!\!\!\!e \cdot \delta_Y ^\ast}&
H_{\bar{C'} \cap T}^0(Y,{j_V}_! \cK_V \otimes \delta_Y ^\ast {R\bar{g'}}_\ast \La) \\
H_{\bar{C}}^0(X \times X,\bar{\cH}) \ar[r]^{\!\!\!\!\!\!\!\!\!\!\!\!\!\!{\rm loc}_{X,C,\cF}}\ar[u]^{(\bar{f} \times \bar{f})^\ast}
&
H_{\bar{C} \backslash C}^0(X \times X,\bar{\mathcal{H}} \otimes {R\bar{g}}_\ast \La)
\ar[u]^{(\bar{f} \times \bar{f})^\ast}\ar[r]^{\!\!\!\!\!e \cdot \delta_X ^\ast}&
H_{\bar{C} \cap S}^0(X,{j}_! \cK_U \otimes \delta_X ^\ast {R\bar{g}}_\ast \La). \ar[u]^{\bar{f} ^\ast}\\
}
\] The assertion follows from this commutative diagram, 
${j_{C'}}_!(f \times f)^\ast u=(\bar{f} \times \bar{f})^\ast
({j_C}_!u)$ by \cite[in the proof of Proposition 2.1.9]{AS} and Definition \ref{refined}.
\end{proof}

We keep the same notation as above. 
Let $\delta_U:U \longrightarrow U \times U$
denote the diagonal map.
In the following, 
we consider the case where $C$ is the diagonal $\delta_U(U).$ 
Further, we assume that $f:V
\longrightarrow U$ is a finite Galois \'{e}tale  morphism of
Galois group $G$. Let $u$ be an endomorphism of $\cF$. Given $\sigma
\in G$, let $\Gamma_\sigma \subset V \times V$ be the graph of $\sigma:V \longrightarrow V.$
Then we have $V \times_{U} V= \coprod_{\sigma \in G}
{\Gamma_\sigma}, $ since $f$ is a finite  Galois \'{e}tale morphism.
Let $\bar{\Gamma}_\sigma$  be the closure of $\Gamma_\sigma$ in
$Y \times Y.$ Let $j_{\sigma}:\Gamma_{\sigma} \longrightarrow
\bar{\Gamma}_\sigma$ denote the open immersion. For $\sigma \in G$,
let $\sigma^\ast:\sigma^\ast f^\ast \cF \longrightarrow f^\ast \cF$
be the canonical map. We consider the composite $f^\ast(u) \circ
\sigma ^\ast:\sigma^\ast f^\ast \cF \longrightarrow f^\ast \cF$ as a
cohomological correspondence of $f^\ast \cF$ on the graph
$\Gamma_\sigma \subset V \times V.$ We have the pull-back 
$$\bar{f}^\ast:H_S^0(X,{j}_! \cK_U)
\longrightarrow
H_T^0(Y,{j_V}_! \cK_V).$$

We assume that $\sigma \not=1.$
Let ${\bar{g'}}_{\sigma}:Y \times Y \backslash \bar{\Gamma}_\sigma \longrightarrow Y \times Y$ be the open immersion.
Since the graph $\Gamma_{\sigma}$ is smooth over $k$ and the intersection ${\Gamma_\sigma} \cap V$
in $V \times V$ is empty, we acquire
the following isomorphism by the long exact sequence (\ref{exact})
in the case where $C=\Gamma_{\sigma}$
$$H_{\bar{\Gamma}_\sigma \cap T}^0(Y,{j_V}_! \cK_V) \simeq
 H_{ \bar{\Gamma}_\sigma \cap T}^0(Y,{j_V}_! \cK_V \otimes \delta_Y ^\ast 
{R{{\bar{g'}}_{\sigma}}}_\ast \La).$$ 
By this isomorphism, we obtain a class
{\small
$$C_{T,!}^{00}({j_V}_! \cF_V,
\bar{\Gamma}_\sigma,{j_\sigma}_!(f ^\ast(u) \circ \sigma ^\ast))
:=C_{T,!}^{0}({j_V}_! \cF_V,
\bar{\Gamma}_\sigma,{j_\sigma}_!(f ^\ast(u) \circ \sigma ^\ast))
-{\rm rk}(\cF_V) \cdot C_{T,!}^{0}({j_V}_! \La_V,
\bar{\Gamma}_\sigma,{j_\sigma}_!\sigma ^\ast)$$}
in $H_{ \bar{\Gamma}_\sigma \cap T}^0(Y,{j_V}_! \cK_V) \simeq
 H_{ \bar{\Gamma}_\sigma \cap T}^0(Y,{j_V}_! \cK_V \otimes \delta_Y ^\ast 
{R\bar{g'}_{\sigma}}_\ast \La).$
Note that 
the class $C_{T,!}^{0}({j_V}_! \cF_V,
\bar{\Gamma}_\sigma,{j_\sigma}_!\\(f ^\ast(u) \circ \sigma ^\ast))$
is equal to the refined characteristic class
$C_{!}^{0}({j_V}_! \cF_V,
\bar{\Gamma}_\sigma,{j_\sigma}_!(f ^\ast(u) \circ \sigma ^\ast))$
recalled in subsection 2.2 by Definition \ref{refined}.
\begin{corollary}\label{30}
Let the notation be as above. Further we assume that $X, Y$ are smooth over $k.$
 Then, we have an equality
$$
\bar{f} ^\ast
C_{S,!}^{00}({j}_! \mathcal {F},{j}_!u)
=\sum_{\sigma \in G}C_{T,!}^{00}({j_V}_! \cF_V,
\bar{\Gamma}_\sigma,{j_\sigma}_!(f ^\ast(u) \circ \sigma ^\ast))$$
in
$H_T^0(Y,{j_V}_!\mathcal {K}_V).$
\end{corollary}
\begin{proof}
By Proposition \ref{29} and $(f \times f)^\ast u=\sum_{\sigma \in G}f^\ast (u) \circ \sigma^\ast$ by \cite[the proof of Corollary 2.1.11]{AS}, we obtain an equality
$$
\bar{f}^\ast C_{S,!}^{00}({j}_! \cF,{j}_!u)=
\sum_{\sigma \in G}C_{T,!}^{00}({j_V}_! \cF_V,\bar{\Gamma}_\sigma,{j_\sigma}_!(f ^\ast(u) \circ \sigma ^\ast))
$$
in $H_{T}^{0}(Y,{j_V}_!\cK_V \otimes \delta_Y^\ast R\bar{g'}_\ast
\La).$
Since the canonical map $H_T^0(Y,{j_V}_!\cK_V)
\longrightarrow H_{T}^{0}(Y,{j_V}_!\cK_V \otimes \delta_Y^\ast
R\bar{g'}_\ast \La)$ is injective, the assertion follows from Definition \ref{aho1}.
\end{proof}

\section{Proof of the localized Abbes-Saito formula}

In this section, we give a proof of the localized Abbes-Saito
 formula assuming
the strong resolution of singularities.
Let $X$ be a smooth scheme of dimension $d$ over a perfect field $k$ and $j:U
\longrightarrow X$ be an open immersion with dense image. 
Let $S$ denote the complement $X \backslash U.$
We assume that $l$ denotes a prime number invertible in $k$ and
$E$ denotes a finite extension of $\mathbb{Q}_l$.
Let $\cF$ be a smooth $E$-sheaf on
$U$. 
Let $E_0$
denote $E \cap \mathbb{Q}(\mu_{p^{\infty}}).$
The naive Swan class
${\rm Sw}^{{\rm naive}}(\cF) \in CH_0(S)_{E_0}$
is defined in
\cite[Definition 4.2.2]{KS} and recalled in \cite[subsection 3.2]{AS}.
\begin{theorem}(the localized Abbes-Saito formula)\label{43}
Let the notation and the assumption be as above. Further, we assume the strong resolution of singularities. Then we have
$$C_{S}^{00}({j}_! \cF)
=-{\rm cl}({\rm Sw}^{{\rm naive}}(\cF))
$$ in $H_S^0(X,\cK_X)$
where ${\rm cl}:CH_0(S)_{E_0} \longrightarrow H_S^0(X,\cK_X)$
denotes the cycle class map.
\end{theorem}
\begin{proof}
Let $\mathcal{O}$ be the integer ring of $E$ 
and $\lambda$ the
maximal ideal of $\mathcal{O}.$
For a constructible $E$-sheaf $\cF$ on $X$, 
$\cF_{\mathcal{O}}$ denotes an $\mathcal{O}$-lattice 
and
$\cF_n$ denotes the reduction $\cF_{\mathcal{O}} \otimes_{\mathcal{O}} \mathcal{O}/\lambda^n.$ 
We put $\bar{\cF}=\cF_1$.

We take the
following cartesian diagram
 \[\xymatrix{
V \ar[r]^{j_V} \ar[d]^{f} & Y \ar[d]^{\bar{f}} & D \ar[l] \ar[d] \\
U \ar[r]^{j}& X & S \ar[l]\\
}
\]
where $f$ is a finite Galois \'{e}tale morphism of Galois group $G$
that trivializes the reduction $\bar{\cF} $ and $\bar{f}:Y \longrightarrow X$ is a proper
morphism. Since we assume the strong resolution of singularities, we may assume that $Y$ is smooth over $k$ and 
$D=\bigcup_{i \in I}D_i \subset Y$
is a divisor with simple normal crossings. We put $\cF_V:=f^\ast
\cF$ on $V.$

By Corollary \ref{30}, we have
\begin{equation}\label{eq0}
\bar{f}^\ast C_{S}^{00}({j}_!\cF)=\sum_{\sigma \in G}
C_{D}^{00}({j_V}_! \cF_V, \bar{\Gamma}_\sigma,{j_\sigma}_!\sigma ^\ast)
\end{equation}
in $H_D^0(Y,\cK_Y).$ Since we assume the strong resolution of singularities,
the condition in \cite[Theorem 3.3.1]{AS} is satisfied. Therefore we obtain an equality
 $$C({j_V}_!\cF_V,\bar{\Gamma}_\sigma,{j_\sigma}_!\sigma^\ast)
=-s_{V/U}(\sigma){\rm {\rm Tr}}^{\rm Br}(\sigma:\bar{M})$$ in
$H_{\bar{\Gamma}_\sigma \cap Y}^0(Y,\cK_Y)$ for $\sigma \not=1$
by loc.\ cit.
Since we have $\bar{\Gamma}_\sigma \cap Y
 =\bar{\Gamma}_\sigma \cap D$
for
$\sigma \not=1$, the canonical map
$H_{\bar{\Gamma}_\sigma \cap D}^0(Y,\cK_Y) \longrightarrow
H_{\bar{\Gamma}_\sigma \cap Y}^0(Y,\cK_Y)$
is an isomorphism. By this isomorphism and Definition \ref{refined},
we understand the following equalities
\begin{equation}\label{eq1}
C_{D}^0 ({j_V}_! \cF_V,\bar{\Gamma}_\sigma,{j_{\sigma}}_!\sigma
^\ast) =
C({j_V}_!\cF_V,\bar{\Gamma}_\sigma,{j_{\sigma}}_!\sigma^\ast) =
-s_{V/U}(\sigma){\rm {\rm Tr}}^{\rm Br}(\sigma:\bar{M})
\end{equation}
in $ H_{\bar{\Gamma}_\sigma \cap Y}^0(Y,\cK_Y)=H_{\bar{\Gamma}_\sigma \cap D}^0(Y,\cK_Y)$ where $C_{D}^0 ({j_V}_!
\cF_V,\bar{\Gamma}_\sigma,{j_\sigma}_!\sigma ^\ast)$ denotes the image of
the class
$C_{D,!}^0 ({j_V}_! \cF_V,\bar{\Gamma}_\sigma,{j_\sigma}_!\sigma ^\ast)$
under the canonical map $H_{\bar{\Gamma}_\sigma \cap D}^0(Y,{j_V}_!\cK_V) \longrightarrow H_{\bar{\Gamma}_\sigma \cap D}^0(Y,\cK_Y).$ 
By (\ref{eq0}), Lemma
\ref{GOS-lemma4}
(Here we use the strong resolution of singularities.)
 and (\ref{eq1}), we acquire equalities
$$\bar{f}^\ast C_{S}^{00}({j}_!\cF)=\sum_{\sigma \in G}
C_{D}^{00}({j_V}_! \cF_V, \bar{\Gamma}_\sigma,{j_\sigma}_!\sigma
^\ast)=\sum_{\sigma \in G,\sigma \not= 1} -s_{V/U}(\sigma)({\rm {\rm
Tr}}^{\rm Br}(\sigma:\bar{M})-{\rm rk}(\cF))
$$ in $H_{D}^{0}(Y, \cK_Y).$ Since we have $\sum_{\sigma \in G}s_{V/U}(\sigma)=0$, the following equality
holds
$$\bar{f}^\ast C_{S}^{00}({j}_!\cF)
=\sum_{\sigma \in G}-s_{V/U}(\sigma){\rm {\rm Tr}}^{\rm Br}(\sigma:\bar{M}).$$
Applying the functor $\bar{f}_\ast$, we obtain
$$|G| \cdot C_{S}^{00}({j}_!\cF)=\sum_{\sigma \in G}-\bar{f}_\ast s_{V/U}(\sigma){\rm {\rm Tr}}^{\rm Br}(\sigma:\bar{M})=-|G| \cdot {\rm Sw}^{{\rm naive}}(\cF).$$
Hence we have proved the required assertion.
\end{proof}
\begin{remark}Let the notation be as in Definition \ref{bakayarou}.
We expect that the following equality
$C_S^{{\rm log},00}(j_!\cF)=-{\rm cl}({\rm Sw}^{{\rm naive}}(\cF))$ holds
in $H_S^0(X,\cK_X).$
However we do not know a proof.
In the case where $\cF$ is a sheaf of rank 1
which is clean with respect to the boundary, we prove this equality
in \cite[Corollary 3.10]{T}.
\end{remark}
\begin{remark}Without assuming that $X$ is smooth over $k,$
we will define the localized characteristic class
$C_{S}^{0}({j}_! \cF) \in H_S^0(X,\cK_X)$ in Definition \ref{fin}
and prove the equality
$C_{S}^{00}({j}_! \cF)
=-{\rm cl}({\rm Sw}^{{\rm naive}}(\cF))$ 
in Corollary \ref{ten}.
\end{remark}

\section{Kato-Saito conductor formula in characteristic $p>0$}
In this section, we will prove the compatibility of the (logarithmic) localized
characteristic class with proper push-forward. This is a localized version of the Lefschetz-Verdier trace formula. As a corollary, we will prove the Kato-Saito conductor formula
 in characteristic $p>0.$ Originally the Kato-Saito conductor formula calculates the Swan
conductor of a Galois representation which
appears when we consider an $\ell$-adic sheaf on a proper smooth curve over a discrete valuation field
by the 0-cycle class (Kato 0-cycle class defined in \cite{K2} for a sheaf of rank 1
or Swan class)
on the boundary
which is produced by the wild ramification of 
the $\ell$-adic sheaf. 

We prove the compatibility of the logarithmic localized characteristic class
of a smooth $\La$-sheaf with proper push-forward.
 Let the notation be as in 
Lemma \ref{ll}. We write $\phi$
for the projection $f:(X \times X)' \longrightarrow X \times X$
in this section.
Moreover let $Z$ be a smooth scheme of dimension
$e$, $W$ an open subscheme of $Z$.  Let $\delta_Z:Z \longrightarrow Z \times Z$
and $\delta_W:W \longrightarrow W \times W$ be the diagonal closed immersions, and
$g_Z:Z \times Z \backslash \delta_Z(Z)
\longrightarrow Z \times Z$ and $g_W:W \times W \backslash \delta_W(W)
\longrightarrow W \times W$
the open immersions.
We consider a commutative diagram
\[\xymatrix{
U \ar[r]^{j'}\ar[dr]_{f_U}& V \ar[r]^{j_V}\ar[d]^{f} & X \ar[d]^{\bar{f}} & S \ar[d]\ar[l]\\
 & W \ar[r]^{j_W} & Z & T \ar[l] 
\\}
\]where the squares are cartesian,
$f:V \longrightarrow W$ is a
proper smooth morphism and $\bar{f}:X \longrightarrow Z$ is a
proper morphism.

We consider the following cartesian diagram
\[\xymatrix{
X'' \ar[r]\ar[d] &
(X \times X)' \ar[d]^{\phi}  &
(X \times X)' \backslash X'' \ar[l]_{g''}\ar[d]\\
X \times_{Z} X \ar[r]\ar[d] &
X \times X \ar[d]^{\bar{f} \times \bar{f}}
&
X \times X \backslash X \times _Z X \ar[l]\ar[d]\\
Z \ar[r]^{\!\!\!\!\delta_Z} &
Z \times Z & Z \times Z \backslash \delta_Z(Z) \ar[l]_{\!\!\!\!g_Z}. \\
}
\]
where ${g''}:(X \times X)' \backslash {X''} \longrightarrow
(X \times X)'$ is the open immersion.
Let $V''$ be the intersection $X'' \cap (V \times V)'$ in $(X \times X)',$
and $i''_{V}:V'' \longrightarrow (V \times V)'$
the closed immersion.
Let $\bar{h}:(X \times X)' \longrightarrow Z \times Z$
denote the projection.
We have $X' \subset X''$
and $V''$ is a smooth scheme of codimension $e$ in $(V \times V)'$
since the projection $h:(V \times V)' \longrightarrow
W \times W$ is a smooth morphism.
We consider the cartesian diagram
\begin{align}\label{chu}
\xymatrix{
V'' \ar[r]^{\!\!\!\!\!\!\!\!\!\!\!\!i''_V}\ar[d]^{h_W} &
(V \times V)' \ar[d]^{h} & (V \times V)' \backslash V'' \ar[l]_{g''_V}\ar[d]
\\
W \ar[r]^{\!\!\!\!\!\!\!\!\!\!\!\!\delta_W} &
W \times W & W \times W \backslash W \ar[l]_{g_W}\\}
\end{align}
where $g''_V$ and $g_W$
are the open immersions,
and $h$ and $h_W$ are the projections.

We define a map
\begin{align}\label{7.0}
R{\bar{f}}_\ast ({j_V}_! \cK_V \otimes {\delta'}^\ast Rg'_\ast\La)
\longrightarrow {j_W}_! \cK_W \otimes \delta_{Z}^\ast {Rg_{Z}}_\ast \La.
\end{align}
By the smooth base change theorem and the projection formula, we acquire isomorphisms
$${j_W}_!{Rf}_\ast (\cK_V \otimes {\delta'}_{V}^\ast {Rg''_V}_\ast \La)
\simeq {j_W}_!{Rf}_\ast (\cK_V \otimes f^\ast \delta_{W}^\ast {Rg_W}_\ast \La)
\simeq {j_W}_!{Rf}_\ast \cK_V \otimes \delta_{Z}^\ast {Rg_Z}_\ast \La.$$
Therefore we obtain an isomorphism
\begin{align}\label{7.1}
R{\bar{f}}_\ast {j_V}_! \cK_V \otimes \delta_{Z}^\ast R{g_{Z}}_\ast \La
\simeq
R{\bar{f}}_\ast ({j_V}_! \cK_V \otimes {\delta'}^\ast Rg''_\ast \La).
\end{align}
We define the map (\ref{7.0}) to be the composite of the following maps
{\small
$$
R{\bar{f}}_\ast ({j_V}_! \cK_V \otimes {\delta'}^\ast Rg'_\ast\La)
\longrightarrow
R{\bar{f}}_\ast ({j_V}_! \cK_V \otimes {\delta'}^\ast Rg''_\ast \La) 
\simeq
R{\bar{f}}_\ast {j_V}_! \cK_V \otimes \delta_{Z}^\ast R{g_{Z}}_\ast \La 
\longrightarrow
{j_W}_! \cK_W \otimes \delta_{Z}^\ast R{g_{Z}}_\ast \La$$}
where the first map is induced by the canonical map $Rg'_\ast\La \longrightarrow Rg''_\ast \La$ and the second isomorphism is (\ref{7.1})
and the third map is induced by the proper push-forward ${Rf}_\ast \cK_V
\longrightarrow \cK_W$.
Then the map (\ref{7.0}) induces the proper push-forward
\begin{equation}\label{ppush}
\bar{f}_\ast :H_{S}^0(X,{j_V}_! \cK_V \otimes {\delta'}^\ast Rg'_\ast\La)
\longrightarrow
H_{T}^0(Z,{j_W}_! \cK_W \otimes \delta_{Z}^\ast {Rg_{Z}}_\ast \La).
\end{equation}
\begin{lemma}\label{oopai}
The canonical map 
$$H_{X'' \backslash V''}^0((X \times
X)',\bar{\cH'} \otimes Rg''_\ast \La)
\longrightarrow
H_{X''}^0((X \times X)',\bar{\cH'} \otimes Rg''_\ast \La)$$
is an isomorphism.
\end{lemma}
\begin{proof}
It suffices to show that 
$H_{V''}^i((V \times V)',{\bar{\cH'}}_V \otimes {R{g''_V}}_\ast \La)=0$
for all $i$
by the localization sequence.
By the proper base change theorem and the cartesian diagram (\ref{chu}),
we acquire an isomorphism
$H_{V''}^i((V \times V)',{\bar{\cH'}}_V \otimes {R{g''_V}}_\ast \La)
\simeq
H^i(W,R\delta_{W}^!{R{h}}_\ast ({\bar{\cH'}}_V \otimes {R{g''_V}}_\ast \La)).$
We put $\cH_W:=
R\mathcal{H}om({\rm pr}_2^\ast R{f_U}_! \cF,R{\rm pr}_1^! R{f_U}_! \cF)$
on $W \times W.$
We write $\phi_V$ for the projection $f_V:(V \times V)' \longrightarrow V \times V$
in subsection 3.3.
The isomorphism ${R{\phi_V}}_\ast {\bar{\cH'}}_V \simeq \bar{\cH}_V$
by \cite[Lemma 2.2.4]{AS} and the Kunneth formula
induce an isomorphism $Rh_\ast {\bar{\cH'}}_V \simeq \cH_W.$
Since we have an isomorphism
$h^\ast {Rg_W}_\ast \La \simeq {Rg''_V}_\ast \La$
by the smooth base change theorem,
we acquire an isomorphism 
${R{h}}_\ast ({\bar{\cH'}}_V \otimes {R{g''_V}}_\ast \La) \simeq
\cH_W \otimes {Rg_W}_\ast \La$ by the projection formula.
Since $R^q{f_U}_!\cF$ is a smooth sheaf on $W$
for all $q$,
we obtain the following vanishing
$H^i(W,R\delta_{W}^!{R{h}}_\ast ({\bar{\cH'}}_V \otimes {R{g''_V}}_\ast \La))
\simeq
H^i(W,R\delta_{W}^!(\cH_W \otimes {Rg_W}_\ast \La))=0$
again by the projection formula.
Hence we have proved the required assertion.
\end{proof}

\begin{theorem}(localized Lefschetz-Verdier trace formula)\label{tth}
Let the notation and the assumption be as above. 
Then we have an equality
$$\bar{f}_\ast C_{S}^{\rm log,0}({j}_! \cF)=C_{T}^{0}({j_W}_! {Rf_U}_! \cF)$$
in $H_{T}^0(Z,\cK_Z).$
\end{theorem}
\begin{proof}
We prove the assertion by a similar method to the one in $\cite[\rm Th\acute{e}or\grave{e}me\ 4.4]{Gr}.$
We put $\bar{\cH}_{Z}:=R{\cH}om({\rm pr}_2^\ast{j_W}_! R{f_U}_! \cF,R{\rm pr}_1^! {j_W}_! R{f_U}_! \cF)$ on $Z \times Z$. By Lemma \ref{5} and the Kunneth formula,
we have an isomorphism $R{\bar{h}}_\ast \bar{\cH'} \simeq \bar{\cH}_{Z}.$
We consider the following commutative diagram
\begin{align}\label{dd}
\xymatrix{
{\phi}^\ast {\rm id}_{j_!\cF} \in H_{X'}^0((X \times X)',\bar{\cH'}) \ar[r]\ar[d]^{{\rm can.}} &
H_{X'}^0((X \times X)',\bar{\cH'} \otimes Rg'_\ast\La) \ar[d]^{{\rm can.}}
\\
H_{X''}^0((X \times X)',\bar{\cH'}) \ar[r] &
H_{X''}^0((X \times X)',\bar{\cH'} \otimes Rg''_\ast \La)\\}
\end{align}
\[\xymatrix{
 & H_{X' \backslash V'}^0((X \times X)',\bar{\cH'} \otimes Rg'_\ast\La)
\ar[l]_{\!\!\!\!\!\!\!\!\!\!\!\!\!\!\!\!\!\!\!\!\!\!\!\!\!\!\!\!\!\!\!\!\!\!\!\!\!\!\!\!\!\!\!\!\!\!\!\!\!\!\!\!\!\!\!\!\rm Lemma \ref{7}}^{\!\!\!\!\!\!\!\!\!\!\!\!\!\!\!\!\!\!\!\!\!\!\!\!\!\!\!\!\!\!\!\!\!\!\!\!\!\!\!\!\!\!\!\!\!\!\!\!\!\!\!\!\!\rm inj.}\ar[d]^{{\rm can.}}\ar[r]^{\!\!\!\!\!\!\!\!\!\!\!\!\!\!\!\!\!e' \cdot {\delta'}^\ast } & H_{S}^0(X,{j_V}_!\cK_V \otimes {\delta'}^\ast Rg'_\ast\La) \ni C_{S}^{\rm log,0}(j_!\cF)
\ar[d]^{{\rm can.}} &&&&&&&&&  \\
 & H_{X'' \backslash V''}^0((X \times X)',\bar{\cH'} \otimes Rg''_\ast \La) \ar[l]^{\!\!\!\!\!\!\!\!\!\!\!\!\!\!\!\!\!\!\!\!\!\!\!\!\!\!\!\!\!\!\!\!\!\!\!\!\!\!\!\!\!\!\!\!\!\!\!\!\!\!\!\!\!\simeq}_{\!\!\!\!\!\!\!\!\!\!\!\!\!\!\!\!\!\!\!\!\!\!\!\!\!\!\!\!\!\!\!\!\!\!\!\!\!\!\!\!\!\!\!\!\!\!\!\!\!\!\!\!\!\!\!\rm Lemma \ref{oopai}}\ar[r]^{\!\!\!\!\!\!\!\!\!\!\!\!\!\!\!\!\!\!\!\!\!e' \cdot {\delta'}^\ast} & H_{S}^0(X,{j_V}_!\cK_V \otimes {\delta'}^\ast Rg''_\ast \La) \ni e'{\delta'}^\ast \phi ^\ast ({\rm id}_{j_!\cF})''.&&&&&&&&
}\]
We denote by $e'{\delta'}^\ast \phi ^\ast ({\rm id}_{j_!\cF})''$ the image of the element
$\phi ^\ast {\rm id}_{j_!\cF}$ in $H_{X''}^0((X \times X)',\bar{\cH'})$
by the composite of the maps in the lower line in the above diagram.
By the above commutative diagram, Lemma \ref{existence} and Definition \ref{bakayarou},
 we obtain an equality 
 \begin{equation}\label{chinko}
 C_{S}^{\rm log,0}(j_!\cF)=e'{\delta'}^\ast \phi ^\ast ({\rm id}_{j_!\cF})''
\end{equation} 
in
$H_{S}^0(X,{j_V}_!\cK_V \otimes {\delta'}^\ast Rg''_\ast \La)$
where we denote 
by the same letter
$C_{S}^{\rm log,0}(j_!\cF)$ the image
of $C_{S}^{\rm log,0}(j_!\cF) \in H_{S}^0(X,{j_V}_!\cK_V \otimes {\delta'}^\ast Rg'_\ast\La)$
by the canonical map
$H_{S}^0(X,{j_V}_!\cK_V \otimes {\delta'}^\ast Rg'_\ast\La)
\longrightarrow
H_{S}^0(X,{j_V}_!\cK_V \otimes {\delta'}^\ast Rg''_\ast \La).$

We consider the following commutative diagram
\begin{align}\label{cc}
{\small
\xymatrix{
H_{X''}^0((X \times X)',\bar{\cH'}) \ar[r]^{\!\!\!\!\!\!\!\!\!\!\!\!\!\!\!\!\!\!\!\!\!\!\!\!(0)}\ar[d]^{{\bar{h}}_\ast}_{\simeq}  &
H_{X'' \backslash V''}^0((X \times X)',\bar{\cH'} \otimes Rg''_\ast \La) \ar[d]^{{\bar{h}}_\ast}_{\simeq}
\ar[r]^{e' \cdot {\delta'}^\ast\!\!\!\!\!\!\!\!} &
H_{S}^0(X,{j_V}_!\cK_V \otimes {\delta'}^\ast Rg''_\ast \La)\ar[d]^{{\bar{f}}_\ast}_{\simeq} \\
H_Z^0(Z\times Z,\bar{\cH}_Z)\ar[r]^{\!\!\!\!\!\!\!\!\!\!\!\!\!\!\!\!\!\!\!\!\!\!\!\!\!\!\!\!\!\!(1)}\ar[dr]_{\!\!\!\!\!\!\!\!\!\!\!\!\!(2)} &
H_{Z \backslash W}^0(Z \times Z,R{\bar{h}}_\ast (\bar{\cH'} \otimes Rg''_\ast \La)) \ar[r]^{(1)'}&
H_T^0(Z,R{\bar{f}}_\ast ({j_V}_!\cK_V \otimes {\delta'}^\ast Rg''_\ast \La)) \\
 &
H_{Z \backslash W}^0(Z \times Z,\bar{\cH}_Z \otimes R{g_Z}_\ast \La) \ar[u]_{(2)''}\ar[r]^{\!\!\!\!\!\!(2)'} \ar[dr]_{(3)'} &
H_T^0(Z,R{\bar{f}}_\ast {j_V}_!\cK_V \otimes \delta_{Z}^\ast R{g_Z}_\ast \La)\ar[u]_{\simeq(\ref{7.1})} \ar[d]^{\rm can.} \\
 &  &
H_T^0(Z,{j_W}_!\cK_W \otimes \delta_{Z}^\ast R{g_Z}_\ast \La).
\\}}
\end{align}
We explain the maps and the commutativities in the above diagram.
The commutativities except for the bottom
one follow from definitions of the maps immediately.
\\(0): This map is the composite of the first two maps in the lower line in the diagram (\ref{dd}). 
\\(1):
The adjoint map ${{\bar{h}}}^\ast \bar{\cH}_Z \longrightarrow \bar{\cH'}$
of the isomorphism $\bar{\cH}_Z \longrightarrow {R{\bar{h}}}_\ast \bar{\cH'}$
and the canonical map $\La \longrightarrow Rg''_\ast \La$
induce a map ${{\bar{h}}}^\ast \bar{\cH}_Z \longrightarrow \bar{\cH'}
\otimes Rg''_\ast \La$.
The adjoint $\bar{\cH}_Z \longrightarrow R{\bar{h}}_\ast ({\bar{\cH'}}
\otimes Rg''_\ast \La)$ of this map
induces a map $H_Z^0(Z \times Z,\bar{\cH}_Z) \longrightarrow
 H_{Z}^0(Z \times Z,R{\bar{h}}_\ast ({\bar{\cH'}} \otimes Rg''_\ast \La)).$
By Lemma \ref{oopai}, the canonical map
$H_{Z \backslash W}^0(Z \times Z,R{\bar{h}}_\ast ({\bar{\cH'}}
\otimes Rg''_\ast \La)) \longrightarrow
H_{Z}^0(Z \times Z,R{\bar{h}}_\ast ({\bar{\cH'}} \otimes Rg''_\ast \La)) 
$ is an isomorphism. The map $(1)$ is the composition of these maps.
\\$(1)'$:
The base change map 
$\delta_{Z}^\ast R{\bar{h}}_\ast \longrightarrow 
R{\bar{f}}_\ast {\delta'}^\ast$ induces a map
$\delta_{Z}^\ast R{\bar{h}}_\ast (\bar{\cH'} \otimes Rg''_\ast \La)
\longrightarrow
R{\bar{f}}_\ast ({\delta'}^\ast \bar{\cH'} \otimes {\delta'}^\ast Rg''_\ast \La).$ 
The evaluation map (\ref{6.2}) induces a map
$R{\bar{f}}_\ast ({\delta'}^\ast \bar{\cH'} \otimes {\delta'}^\ast Rg''_\ast \La)
\longrightarrow
R{\bar{f}}_\ast ({j_V}_!\cK_V \otimes {\delta'}^\ast Rg''_\ast \La)$. 
We define $(1)'$ to be the composite of these two maps.
By these definitions, the commutativities in the first line in the diagram are clear.

(2):
This map is the map ${\rm loc}_{Z,W,{R{f_U}}_!\cF}:H_Z^0(Z \times Z,{\bar{\cH}}_Z) \longrightarrow
H_{Z}^0(Z \times Z,{\bar{\cH}}_Z \otimes {Rg_Z}_\ast \La)
\simeq 
H_{Z \backslash W}^0(Z \times Z,{\bar{\cH}}_Z \otimes {Rg_Z}_\ast \La).$
(c.f.\ (\ref{loc map}).)
\\$(2)'$: 
The isomorphism ${R{\bar{h}}}_\ast \bar{\mathcal{H'}} \simeq {\bar{\cH}}_Z$
and
the base change map 
$\delta_{Z}^\ast R{\bar{h}}_\ast \longrightarrow R{\bar{f}}_\ast {\delta'}^\ast$ 
induce a map
$\delta_{Z}^\ast \bar{\cH}_Z \otimes \delta_{Z}^\ast R{g_Z}_\ast \La
\longrightarrow
 R{\bar{f}}_\ast {\delta'}^\ast \bar{\cH'} \otimes \delta_{Z}^\ast R{g_Z}_\ast \La.$
The evaluation map (\ref{6.2}) induces a map
$R{\bar{f}}_\ast {\delta'}^\ast \bar{\cH'} \otimes \delta_{Z}^\ast R{g_Z}_\ast \La
\longrightarrow
R{\bar{f}}_\ast {j_V}_!\cK_V \otimes R{g_Z}_\ast \La.$
We define $(2)'$ to be the composite of these maps.
\\$(2)''$: 
The map ${{\bar{h}}}^\ast {\bar{\cH}}_{Z} \longrightarrow {\bar{\cH'}}$ and the base change map 
${{\bar{h}}}^\ast {R{g_Z}}_\ast \La \longrightarrow
{Rg''}_\ast \La$ induce a map
${{\bar{h}}}^\ast ({\bar{\cH}}_{Z} \otimes {R{g_Z}}_\ast \La)
\longrightarrow
{\bar{\cH'}} \otimes {Rg''}_\ast \La$. The map $(2)''$ is induced by the adjoint of this map.
By these descriptions and the definition of the map (\ref{7.1}),
the commutativities in the second line in the diagram (\ref{cc})
follow.

$(3)'$: This map is induced by the pull-back by $\delta_Z$
and the usual evaluation map
$\delta_{Z}^\ast {\bar{\cH}}_{Z} \longrightarrow {j_W}_!\cK_W$ for ${Rf_U}_!\cF.$
The right bottom commutativity is a consequence of the compatibility
of evaluation maps with proper push-forward which is proved in $\cite[\rm Th\acute{e}or\grave{e}me\ 4.4.\ (4.4.4)]{Gr}.$

We consider the following commutative diagram
\[\xymatrix{
\phi^\ast {\rm id}_{j_!\cF} \in H_{X''}^0((X \times X)',\bar{\cH'}) \ar[r]^{\phi_\ast\!\!\!\!\!}_{\simeq\!\!\!\!}\ar[d]^{\simeq}_{\bar{h}_\ast} & 
H_{X \times _Z X}^0(X \times X,\bar{\cH}) \ni {\rm id}_{j_!\cF} \ar[dl]_{\simeq}^{(\bar{f} \times \bar{f})_\ast} \\
{\rm id}_{{j_W}_!{Rf_U}_!\cF} \in H_{Z}^0(Z \times Z,\bar{\cH}_Z). & \\
}
\]
By this diagram, an equality
$(\bar{f} \times \bar{f})_\ast {\rm id}_{j_!\cF}={\rm id}_{{j_W}_!{Rf_U}_!\cF}$
which is a consequence of the compatibility of the cohomological correspondence with proper
push-forward
and ${\phi}_\ast {\phi}^\ast ={\rm id},$
we obtain an equality 
\begin{equation}\label{kuso}
\bar{h}_\ast {\phi}^\ast {\rm id}_{j_!\cF}={\rm id}_{{j_W}_!{Rf_U}_!\cF}.
\end{equation}
We consider the following commutative diagram
\begin{align}\label{hentai}
\xymatrix{
H_{S}^0(X,{j_V}_! \cK_V \otimes {\delta'}^\ast Rg'_\ast\La)
\ar[r]^{\!\!\!\!\!\!\!\!\!\!\!\!\!\!\!\!\!\!\!\rm can.}\ar[d]^{{\bar{f}}_\ast} &
H_{S}^0(X,\cK_X \otimes {\delta'}^\ast Rg'_\ast\La)
\simeq H_S^0(X,\cK_X)
\ar[d]^{{\bar{f}}_\ast} \\
H_{T}^0(Z,{j_W}_! \cK_W \otimes  \delta_{Z}^\ast {Rg_{Z}}_\ast \La)
\ar[r]^{\!\!\!\!\!\!\!\!\!\!\!\!\!\!\!\rm can.} &
H_{T}^0(Z,\cK_Z \otimes \delta_{Z}^\ast {Rg_{Z}}_\ast \La)
\simeq H_T^0(Z,\cK_Z)\\
}
\end{align}
where the left vertical arrow is the proper push-forward (\ref{ppush})
and the right vertical arrow
${\bar{f}}_\ast:H_S^0(X,\cK_X) \longrightarrow H_T^0(Z,\cK_Z)$
is the usual proper push-forward. 
Clearly the composite $H_{S}^0(X,{j_V}_! \cK_V \otimes {\delta'}^\ast Rg'_\ast\La)
\longrightarrow
H_{T}^0(Z,{j_W}_! \cK_W \otimes \delta_{Z}^\ast {Rg_{Z}}_\ast \La)$
of the right vertical arrows in the diagram (\ref{cc})
is equal to the map ${\bar{f}}_\ast:H_{S}^0(X,{j_V}_! \cK_V \otimes {\delta'}^\ast Rg'_\ast\La)
\longrightarrow
H_{T}^0(Z,{j_W}_! \cK_W \otimes \delta_{Z}^\ast {Rg_{Z}}_\ast \La)$
in the diagram (\ref{hentai}).
The localized characteristic class
$C_{T}^{0}({j_W}_! {Rf_U}_! \cF)$
is the image of the cohomological correspondence
${\rm id}_{{j_W}_!{Rf_U}_!\cF}$
by the composite of the maps $(2)$ and $(3)'$ in the diagram (\ref{cc}) by Definition \ref{refined}.
Hence the assertion follows from the equalities (\ref{chinko}) and (\ref{kuso}),
and the commutative diagrams (\ref{cc}) and (\ref{hentai}).
\end{proof}
\begin{corollary}\label{Coo}
Let the notation and the assumption be as in Theorem \ref{tth}. 
Then we have an equality
$$\bar{f}_\ast C_{S}^{\rm log,00}({j}_! \cF)=C_{T}^{00}({j_W}_! {Rf_U}_! \cF)
-{\rm rk}(\cF) \cdot C_{T}^{00}({j_W}_!{Rf_U}_!\La_{U})$$
in $H_{T}^0(Z,\cK_Z).$
\end{corollary}
\begin{proof}
By Theorem \ref{tth},
we have equalities 
$\bar{f}_\ast C_{S}^{\rm log,0}({j}_! \cF)=C_{T}^{0}({j_W}_! {Rf_U}_! \cF)$ and 
$\bar{f}_\ast C_{S}^{\rm log,0}({j}_! \La_U)=C_{T}^{0}({j_W}_! {Rf_U}_! \La_U).$
Hence the assertion follows from an equality ${\rm rk}({Rf_U}_! \cF)={\rm rk}(\cF) \cdot {\rm rk}({Rf_U}_! \La_U).$
\end{proof}
\begin{corollary}\label{rank1}
Let the notation be as in Theorem \ref{tth}.
Further we assume that $k$ is a perfect field, that $D \cup S$
is a divisor with simple normal crossings, that dim $Z \leq 2$
and that $\cF$ is a smooth $E$-sheaf
of rank 1 which is clean with respect to the boundary
where $E$ is a finite extension of $\mathbb{Q}_l$
and $l$ is invertible in $k.$
Then we have
$$-{\bar{f}}_\ast c_{\cF}={\rm Sw}^{\rm naive} (R{f_U}_! \cF)-{\rm rk}(\cF) \cdot {\rm Sw}^{\rm naive} (R{f_U}_! \La_U)$$
in $H_T^0(Z,\cK_Z)$
where $c_{\cF} \in CH_0(S)$ is the Kato 0-cycle class defined by K.\ Kato
in \cite{K1} and \cite{K2}, and recalled in \cite[Section 4]{AS}.
\end{corollary}
\begin{proof}
We prove an equality $-c_{\cF}=C_S^{\rm log,00}(j_!\cF)$ in
$H_S^0(X,\cK_X)$ in \cite[Corollary 3.10]{T}.
By this equality, Theorem \ref{43} and Corollary \ref{Coo}, the assertion follows.
\end{proof}

\begin{remark}Let the notation be as in Corollary \ref{rank1}.
If we assume the strong resolution of singularities, the equality
$$-{\bar{f}}_\ast c_{\cF}={\rm Sw}^{\rm naive} (R{f_U}_! \cF)-{\rm rk}(\cF) \cdot {\rm Sw}^{\rm naive} (R{f_U}_! \La_U)$$
in $H_T^0(Z,\cK_Z)$ holds for {\it any} dimensional scheme $Z.$
\end{remark}

We prove the compatibility of the localized characteristic class of a smooth $\La
$-adic sheaf with a cohomological correspondence
with proper push-forward.
Let $X$ and $Y$ be schemes over $k$ and $U \subseteq X$ and $V \subseteq
Y$ open dense subschemes smooth over $k$ respectively. Let $j_V:V
\longrightarrow Y$ and $j:U \longrightarrow X$ denote the open
immersions, and $T:=Y \backslash V$ and $S:=X \backslash U$ the complements respectively. 
Let $\delta_{X}:X \longrightarrow X \times X,$
$\delta_{Y}:Y \longrightarrow Y \times Y$, $\delta_V:V
\longrightarrow V \times V$ and $\delta_U:U \longrightarrow U \times
U$ be the diagonal closed immersions.  We consider a cartesian diagram
\[\xymatrix{
V \ar[r]^{j_V} \ar[d]^{f} & Y \ar[d]^{\bar{f}} & T \ar[l] \ar[d] \\
U \ar[r]^{j} & X & S \ar[l] \\
}
\]where
 $\bar{f}:Y \longrightarrow
X$ is a proper morphism and $f:V \longrightarrow U$ is a proper
smooth morphism.

Let $C$ and $C'$ be closed subschemes
of $U \times U$ and $V \times V$ respectively.
Let $\bar{C}$ be the closure of $C$ in $X \times X$ and
$\bar{C'}$ the closure of $C'$ in $Y \times Y$ respectively.
We assume that $(f \times f)(C') \subset C$ and
$C=(X \times U) \cap \bar{C}$ and $C'=(Y \times V) \cap \bar{C'}$.
Let $C''$ denote the inverse $(f \times f)^{-1}(C)$ and $\bar{C}''$
 the closure of $C''$ in $Y \times Y$.
 Let $j_C:C \longrightarrow \bar{C}$ and $j_{C'}:C' \longrightarrow \bar{C'}$ be the open immersions.
We consider the following cartesian diagram
\[\xymatrix{
\bar{C'} \ar[r]\ar[dr] & \bar{C}'' \ar[r]^{\!\!\!\!\!\!\!\!\!\bar{c}''}\ar[d]&
Y \times Y \ar[d]^{\bar{f} \times \bar{f}} &
Y \times Y \backslash \bar{C}'' \ar[l]_{g_{\bar{C}''}}\ar[d] \\
 & \bar{C} \ar[r]^{\!\!\!\!\!\!\!\!\!\!\!\!\bar{c}}&
X \times X &
X \times X \backslash \bar{C} \ar[l]_{g_{\bar{C}}}
}
\]where $g_{\bar{C}}:X \times X \backslash \bar{C} \longrightarrow X \times X$ and 
$g_{\bar{C}''}:Y \times Y \backslash \bar{C}'' \longrightarrow Y \times Y$ are the open immersions
and the squares are cartesian.
Similarly
$g_{C}:U \times U \backslash C \longrightarrow U \times U,$
$g_{C'}:V \times V \backslash C' \longrightarrow V \times V$
and $g_{\bar{C'}}:Y \times Y \backslash \bar{C'} \longrightarrow Y \times Y$ 
denote the open immersions.
Let $h:C' \longrightarrow C$ be the projection.

Let $\mathcal{F}$ be a smooth $\Lambda$-sheaf on $V$. Since $f$ is a proper smooth morphism, the sheaves $R^q f_\ast \cF$ are smooth
for all $q$.
Let $u'$ be a cohomological correspondence of $\cF$ on $C'$.
We put $\bar{\cH}:=R\mathcal{H}om({\rm pr}_2^\ast{j}_!Rf_\ast \cF,R{\rm pr}_1^!{j}_!Rf_\ast \cF)$ on $X \times X$ and
$\bar{\cH'}:=R\mathcal{H}om({\rm pr}_2^\ast{j_V}_!\cF,R{\rm pr}_1^!{j_V}_!\cF)$ on $Y \times Y$ respectively.

We define a map in the same way as (\ref{7.0})
\begin{equation}\label{ppush3}
R\bar{f}_\ast ({j_V}_!\cK_V \otimes \delta_{Y}^\ast {Rg_{\bar{C'}}}_\ast \La)
\longrightarrow
{j}_!\cK_U \otimes \delta_{X}^\ast
{Rg_{\bar{C}}}_\ast \La
\end{equation}
to be the composite of the following maps
$$
R{\bar{f}}_\ast ({j_V}_! \cK_V \otimes \delta_{Y}^\ast {Rg_{\bar{C'}}}_\ast \La)\longrightarrow
R{\bar{f}}_\ast ({j_V}_! \cK_V \otimes \delta_{Y}^\ast {Rg_{\bar{C}''}}_\ast \La) \simeq 
R{\bar{f}}_\ast {j_V}_! (\cK_V \otimes \delta_{V}^\ast R{g_{C''}}_\ast \La) $$
$$\simeq R{\bar{f}}_\ast {j_V}_! \cK_V \otimes \delta_{X}^\ast R{g_{\bar{C}}}_\ast \La \longrightarrow
{j}_! \cK_U \otimes \delta_{X}^\ast R{g_{\bar{C}}}_\ast \La.$$
The first map is induced by the canonical map ${Rg_{\bar{C'}}}_\ast \La
\longrightarrow {Rg_{\bar{C}''}}_\ast \La.$
The second isomorphism is induced by the projection formula.
The third isomorphism follows from the smooth base change theorem.
The fourth map is induced by the proper push-forward $R{\bar{f}}_\ast {j_V}_! \cK_V
\longrightarrow {j}_! \cK_U.$
The map (\ref{ppush3}) induces the proper push-forward 
\begin{equation}\label{pus}
\bar{f}_\ast :H_{\bar{C'} \cap T}^0(Y,{j_V}_!\cK_V \otimes \delta_{Y}^\ast {Rg_{\bar{C'}}}_\ast \La) 
\longrightarrow 
H_{\bar{C} \cap S}^0(X,{j}_!\cK_U \otimes \delta_{X}^\ast
{Rg_{\bar{C}}}_\ast \La).
\end{equation}
\begin{proposition}\label{prop:2}
Let the notation and the assumption be as above. Then we have
$$\bar{f}_\ast C_{T,!}^0({j_V}_!\cF,\bar{C'},{j_{C'}}_!u')
=C_{S,!}^0({j}_!Rf_\ast \cF,\bar{C},{j_C}_!h_\ast u')$$
in
$H_{\bar{C} \cap S}^0(X,{j}_!\cK_U \otimes \delta_{X}^\ast
{Rg_{\bar{C}}}_\ast \La).$
\end{proposition}

\begin{proof} We prove this formula in the same way as Theorem \ref{Coo}.
We omit a proof.
\end{proof}

We keep the same notation as above. In the following, we consider the case where
$C=\delta_U(U)$ and $C'=\delta_V(V)$ are the diagonals and $u'={\rm id}_{\cF}$.
We assume that $X, Y$ are smooth over $k.$
We have the proper push-forward 
$\bar{f}_\ast:H_T^0(Y,{j_V}_!\cK_V) \longrightarrow H_{S}^0(X,{j}_!\cK_U).$
\begin{corollary}
\label{cor:push-forward}
Let the notation and the assumption be as above. 
\\1. We have
$$\bar{f}_\ast C_{T,!}^{00}({j_V}_!\cF)=
C_{S,!}^{00}({j}_!Rf_\ast {\cal{F}})-{\rm {\rm rk}}(\cF) \cdot C_{S,!}^{00}({j}_!Rf_\ast {\La}_V) $$
in $H_{S}^0(X,{j}_!{\cK_U}).$
\\2. We keep the same notation as in 1. Then we have an equality
$$\bar{f}_\ast C_{T}^{00}({j_V}_!\cF)=
C_{S}^{00}({j}_!Rf_\ast {\cal{F}})-{\rm {\rm rk}}(\cF) \cdot C_{S}^{00}({j}_!Rf_\ast {\La}_V) $$
in $H_{S}^0(X,\cK_X).$
\end{corollary}
\begin{proof}
1. We consider the commutative diagram
\[\xymatrix{
H_{T}^0(Y,{j_V}_!\cK_V) \ar[r]\ar[d]^{\bar{f}_\ast}
&
H_{T}^0(Y,{j_V}_! \cK_V \otimes \delta_Y^\ast R{g_Y}_\ast \La) \ar[d]^{\bar{f}_\ast}\\
H_S^0(X,{j}_!\cK_U) \ar[r] &
H_S^0(X,{j}_!\cK_U \otimes \delta_{X}^\ast R{g_X}_\ast \La)
}
\] where the right vertical arrow
is the map (\ref{pus}) in the case where
$C=\delta_U(U)$ and $C'=\delta_V(V)$ are the diagonals.
Since the canonical map $H_S^0(X,{j}_!\cK_U)  \longrightarrow H_S^0(X,{j}_!\cK_U \otimes \delta_{X}^\ast R{g_X}_\ast \La)$ is injective,    we may regard the equality as an equality in $H_S^0(X,{j}_!\cK_U \otimes \delta_{X}^\ast R{g_X}_\ast \La)$. 
Hence the assertion follows from Proposition \ref{prop:2}
and Lemma \ref{lem:1}.
\\2. The assertion follows from 1 and Remark \ref{aho2} immediately.
\end{proof}
\begin{corollary}\label{COroo}
(Kato-Saito conductor formula in characteristic $p>0$)
Let the notation be as in Corollary \ref{cor:push-forward}. 
We assume that $k$ is a perfect field, that
$E$ denotes a finite extension of $\mathbb{Q}_l$
and that $\cF$ is a smooth $E$-sheaf and
the strong resolution of singularities.
Then we have an equality
$$-{\bar{f}}_\ast {\rm Sw}^{\rm naive}(\cF)={\rm Sw}^{\rm naive} (Rf_\ast \cF)-{\rm rk}(\cF) \cdot {\rm Sw}^{\rm naive} (Rf_\ast E)$$
in $H_S^0(X,\cK_X).$
\end{corollary}
\begin{proof}
This follows from Theorem \ref{43} and Corollary \ref{cor:push-forward}.2.
\end{proof}

Let $X$ be a scheme and $U \subset X$ an open dense subscheme smooth of dimension $d$ over $k.$
Let $S$ be the complement $X \backslash U.$
Let $j:U \longrightarrow X$ denote the open immersion.

Let $\cF$ denote a smooth $\Lambda$-sheaf on $U.$
Assuming the strong resolution of singularities, we define
the localized characteristic class
$C_S^{0}(j_!\cF) \in H_S^0(X,\cK_X)$ and prove the equality 
$C_{S}^{00}(j_!\cF)=-{\rm cl}({\rm Sw}^{\rm naive}(\cF))$ in $H_S^0(X,\cK_X).$
Let $f:X' \longrightarrow X$ be a desingularization preserving the open subscheme $U$
by the assumption of the strong resolution of singularities.
Let $j':U \longrightarrow X'$ denote the open immersion and $S'$ the complement $X' \backslash U.$
We denote by $C_{X'}^0(j_!\cF) \in H_S^0(X,\cK_X)$
the image of the localized characteristic class
$C_{S'}^0(j'_!\cF) \in H_{S'}^0(X',\cK_{X'})$
in Remark \ref{aho2}
 by the proper push-forward
$f_\ast :H_{S'}^0(X',\cK_{X'}) \longrightarrow H_S^0(X,\cK_X).$

\begin{corollary}\label{chinpo}
Let the notation be as above.
We consider two desingularizations $X' \longrightarrow X$ and $X'' \longrightarrow X$
preserving the open subscheme $U.$
Then we have an equality
$$C_{X'}^0(j_!\cF)=C_{X''}^0(j_!\cF)$$
in $H_S^0(X,\cK_X).$
\end{corollary}
\begin{proof}We take a smooth model $\widetilde{X} \longrightarrow X' \times_{X} X''$
preserving $U$ by the strong resolution of singularities.
We consider the following cartesian diagram
\[\xymatrix{
U \ar[d]^{\rm id}\ar[r]^{\widetilde{j}}& \widetilde{X}\ar[d]^{\pi} \\
U \ar[r]^{j'}& X' \\
}
\]
where the right vertical arrow is the canonical projection.
Let $\widetilde{S}$ denote the complement $\widetilde{X} \backslash U.$
It suffices to show
$\pi _\ast C_{\widetilde{S}}^0(\widetilde{j}_!\cF)=C^0_{S'}(j'_!\cF)$
in $H_{S'}^0(X',\cK_{X'}).$
This equality follows immediately from
Corollary \ref{cor:push-forward}.2.
Hence the required assertion follows. 
\end{proof}
\begin{definition}\label{fin}
Let the notation be as in Corollary \ref{chinpo}.
We take a smooth model $f:X' \longrightarrow X$
preserving $U.$
We put $C_S^0(j_!\cF):=C_{X'}^0(j_!\cF) \in H_S^0(X,\cK_X).$ 
This class is independent of a choice of a smooth model
$f:X' \longrightarrow X$ by Corollary \ref{chinpo}.
We call it 
{\it the localized characteristic class of $j_!\cF$}.
Further we put $C_S^{00}(j_!\cF):=C_S^0(j_!\cF)-{\rm rk}(\cF) \cdot C_S^0(j_!\La) \in H_S^0(X,\cK_X).$
\end{definition}
\begin{corollary}\label{ten}
Let the notation be as in Theorem \ref{43}.
We do not assume that $X$ is smooth over $k.$
Then we have
$$C_{S}^{00}(j_!\cF)=-{\rm cl}({\rm Sw}^{\rm naive}(\cF))$$
in $H_S^0(X,\cK_X)$ where the left hand side is the localized characteristic class
defined in Definition \ref{fin}.
\end{corollary}
\begin{proof}
We take a smooth model $f:X' \longrightarrow X$
preserving $U$ by the strong resolution of singularities.
We denote by ${\rm Sw}_{X'}^{\rm naive}(\cF) \in CH_0(S')_{E_0}$
the naive Swan class of $\cF$
with respect to $(U,X').$
We have an equality
$f_\ast {\rm Sw}_{X'}^{\rm naive}(\cF)={\rm Sw}^{\rm naive}(\cF)$
in $CH_0(S)_{E_0}.$
By Definition \ref{fin} and Theorem \ref{43},
we acquire the following equalities
$$C_{S}^{00}(j_!\cF)=f_\ast C_{S'}^{00}(j'_!\cF)=-f_\ast {\rm cl}({\rm Sw}_{X'}^{\rm naive}(\cF))=-{\rm cl}({\rm Sw}^{\rm naive}(\cF))$$
in $H_S^0(X,\cK_X).$
Hence the assertion follows.
\end{proof}

\begin{corollary}
Let the notation be as in Corollary \ref{cor:push-forward}. Moreover, we assume the strong resolution of singularities
 and that $f$ is a finite \'{e}tale morphism.
Let ${\rm d}_{V/U}^{\rm log} \in CH_0(S) \otimes \mathbb{Q}$ be the wild discriminant
of $V$ over $U$ defined in (\cite[Definition 4.3.1]{KS}).
We have
$${\bar{f}}_\ast C_T^{00}({j_V}_!\cF)
=C_S^{00}({j}_!f_\ast \cF)+{\rm rk}(\cF) \cdot {\rm cl}({\rm d}_{V/U}^{\rm log})$$ in $H_S^0(X,\cK_X)$.
\end{corollary}
\begin{proof}
By Corollary \ref{cor:push-forward}.2, we obtain an equality
${\bar{f}}_\ast C_T^{00}({j_V}_!\cF)
=C_S^{00}({j}_!f_\ast \cF)-{\rm rk}(\cF) \cdot C_S^{00}({j}_!f_\ast \Lambda_V).$
By Theorem \ref{43},
we acquire
$C_S^{00}({j}_!f_\ast \Lambda_V)
=-{\rm cl}({\rm Sw}(f_\ast \Lambda_V)).$ 
By the definition of the Swan class,
we have ${\rm Sw}(f_\ast \Lambda_V)={\rm d}_{V/U}^{\rm log}.$
Hence the following equalities hold
\begin{align*}
\lefteqn{
{\bar{f}}_\ast C_T^{00}({j_V}_!\cF)=C_S^{00}({j}_!f_\ast \cF)-{\rm rk}(\cF) \cdot C_S^{00}({j}_!f_\ast \Lambda_V)}\\
&=C_S^{00}({j}_!f_\ast \cF)+{\rm rk}(\cF) \cdot {\rm cl}({\rm d}_{V/U}^{\rm log}).
\end{align*}Thus the assertion follows.
\end{proof}

\end{document}